\documentclass[a4paper,11pt,twoside]{article}
\usepackage[left=1.1in,right=1.1in,top=1.5in,bottom=1.5in]{geometry}
\usepackage{lmodern}
\usepackage[utf8]{inputenc}
\usepackage[T1]{fontenc}
\usepackage[english]{babel}
\usepackage{mathtools,amsmath,amssymb,amsfonts,amsthm,mathrsfs}
\usepackage{bbm}
\usepackage[title,titletoc]{appendix}

\usepackage[normalem]{ulem}

\usepackage{pgfplots}
\usepackage{graphicx}
\graphicspath{{Figures/}}
\usepackage{caption}
\usepackage{subcaption}
\usepackage{microtype}
\usepackage{bigints}
\usetikzlibrary{decorations.markings}
\usetikzlibrary{patterns}

\usepackage[tight]{minitoc}
\usepackage{tocloft}
\pgfplotsset{compat=1.12}

\usepackage{booktabs}
\usepackage{multirow}
\usepackage{rotating,tabularx}
\usepackage{adjustbox}
\usepackage{enumitem}

\usepackage{cases}

\definecolor{bostonuniversityred}{rgb}{0.7, 0.0, 0.0}
\definecolor{forestgreen}{rgb}{0.13, 0.45, 0.13}
\definecolor{blue-violet}{rgb}{0.5, 0.1, 0.8}
\definecolor{blue-new}{rgb}{0.2, 0.2, 0.8}
\definecolor{amber}{rgb}{1.0, 0.49, 0.0}
\definecolor{applegreen}{rgb}{0.55, 0.71, 0.0}

\newtheorem{theorem}{Theorem}[section]
\newtheorem{definition}{Definition}[section]
\newtheorem{lemma}[theorem]{Lemma}
\newtheorem{proposition}[theorem]{Proposition}
\newtheorem{corollary}[theorem]{Corollary}

\theoremstyle{definition}

\theoremstyle{definition}
\newtheorem{remark}[theorem]{Remark}
\numberwithin{equation}{section}

\usepackage[colorlinks]{hyperref}
\hypersetup{citecolor=blue}
\usepackage{fancyhdr}
\makeatletter
\def\thanks#1{\protected@xdef\@thanks{\@thanks
        \protect\footnotetext{#1}}}
\makeatother
\renewenvironment{abstract}[1]
{\list{}{\setlength{\leftmargin}{3em}
 \setlength{\rightmargin}{\leftmargin}}\item[]
\textbf{\abstractname.} #1\relax}
{\endlist}
\providecommand{\thanksbody}
{\hspace{-2em}
The authors are members of the \emph{Grup\-po Na\-zio\-na\-le per l'Anali\-si Ma\-te\-ma\-ti\-ca, la Pro\-ba\-bi\-li\-t\`{a} e le lo\-ro Appli\-ca\-zio\-ni} (GNAMPA) of the \emph{Isti\-tu\-to Na\-zio\-na\-le di Al\-ta Ma\-te\-ma\-ti\-ca} (INdAM) and acknowledge financial support from this institution and the PRIN 2022 project \emph{Modeling, Control and
Games through Partial Differential Equations} (D53D23005620006), funded by the European Union - Next Generation EU.
}

\title{\textbf{Wavefronts for a degenerate reaction-diffusion system with application to bacterial growth models}}

\author{Luisa Malaguti, Elisa Sovrano
\thanks{\thanksbody}
}

\begin{document}

\newcommand{\Addresses}{
\bigskip
\small
\noindent Luisa Malaguti\\
Dipartimento di Scienze e Metodi dell'Ingegneria,\\
Universit\`a degli Studi di Modena e Reggio Emilia,\\
Via G. Amendola 2, 42122 Reggio Emilia, Italy\\
email: \href{mailto:luisa.malaguti@unimore.it}{\texttt{luisa.malaguti@unimore.it}}
\medskip

\noindent Elisa Sovrano\\
Dipartimento di Scienze e Metodi dell'Ingegneria,\\
Universit\`a degli Studi di Modena e Reggio Emilia,\\
Via G. Amendola 2, 42122 Reggio Emilia, Italy\\
email: \href{mailto:elisasovrano@unimore.it}{\texttt{elisasovrano@unimore.it}}
}

\maketitle
\thispagestyle{empty}

{\small
\begin{abstract}
\noindent We investigate wavefront solutions in a nonlinear system of two coupled reaction-diffusion equations with degenerate diffusivity:
\[n_t = n_{xx} - nb, \quad b_t = [D nbb_x]_x + nb,\]
where $t\geq0,$ $x\in\mathbb{R}$, and $D$ is a positive diffusion coefficient. This model, introduced by Kawasaki et al. (J. Theor. Biol. 188, 1997), describes the spatial-temporal dynamics of bacterial colonies $b=b(x,t)$ and nutrients $n=n(x,t)$ on agar plates. Kawasaki et al. provided numerical evidence for wavefronts, leaving the analytical confirmation of these solutions an open problem. We prove the existence of an infinite family of wavefronts parameterized by their wave speed, which varies on a closed positive half-line. We provide an upper bound for the threshold speed and a lower bound for it when $D$ is sufficiently large. The proofs are based on several analytical tools, including the shooting method and the fixed-point theory in Fréchet spaces, to establish existence, and the central manifold theorem to ascertain uniqueness.


\end{abstract}
}

\section{Introduction and main results}\label{section-1}

In this article, we are exploring a nonlinear system of two coupled reaction-diffusion equations that models the growth of bacteria on agar plates with nutrients. Based on laboratory experiments and numerical analyses, it has been observed that bacterial evolution can produce various spatiotemporal patterns during bacterial growth, including rings, disks, and dense branching morphologies, which depend on the bacterial species and nutrient conditions (see~\cite{GKCB-98,Ha-04,KMUS-97,MSM-00}). To study the evolution of bacteria and nutrients, we consider the reaction-diffusion model as described in \cite{KMUS-97}:
\begin{subnumcases}{\label{eq-sys}}
	n_t= D_n n_{xx}-nb, \\
	b_t=[D_b nbb_x]_x+nb,
\end{subnumcases}
where $n=n(x,t)$ and $b=b(x,t)$ are the concentration of the nutrient and the population density of bacteria, respectively, at time $t\geq0$ and space $x\in\mathbb{R}$. The parameters $D_n$ and $D_b$ are the nutrient and bacteria diffusion coefficients, respectively. The diffusivity of bacteria is nonlinear and degenerates to model its trend toward reduced mobility when the nutrient or the bacteria concentration level is low (i.e., if either $n$ or $b$ tends to zero).

Experimental observations reported in~\cite{KMUS-97} have shown that system~\eqref{eq-sys} can generate spatiotemporal patterns supported by traveling wave solutions. These solutions are characterized by profiles that propagate at a constant speed, specifically represented by:
\begin{equation}\label{wf}
n(x,t) = \eta(x - ct), \quad b(x,t) = \beta(x - ct),
\end{equation}
where $(\eta,\beta)$ represents the profile and $c \in \mathbb{R}$ denotes the wave speed, which is an unknown of the problem. Despite these experimental observations, there is a lack of theoretical proof for the existence of traveling wave solutions for  system~\eqref{eq-sys}, nor has there been an analytical study of the admissible wave speeds. This work aims to address these theoretical and analytical gaps.

In~\cite{SMGA-01}, researchers analyzed the case $D_n=0$ and, through both numerical and analytical approaches, indicated that wavefronts display ``sharp'' behavior in the bacteria profile at a critical speed (see Definition~\ref{def-2}). In~\cite{MHST-24}, the studied system is
\[
\begin{cases}
n_t= -f(n,b), \\
b_t=[g(n)h(b)b_x]_x+f(n,b),
\end{cases}
\]
where the diffusion term $g(n)h(b)$ degenerates at both $n = 0$ and $b = 0$, and the reaction term $f(n,b)$ is positive except when $n = 0$ or $b = 0$. They established the existence of wavefronts for all wave speeds within a closed half-line and, similar to the model in~\cite{SMGA-01}, showed that the profile exhibits ``sharp'' behavior at the threshold speed. It has been also demonstrated in~\cite{MHST-24} that the profiles for all other speeds are smooth. For the classical non-degenerate case, the analogous system studied in literature is:
\[
\begin{cases}
n_t= \lambda n_{xx}-f(n,b), \\
b_t=b_{xx}+f(n,b),
\end{cases}
\]
where $\lambda \geq 0$. Results depend on different properties of the function $f$, which is consistently assumed to be as before (see, e.g.,~\cite{AiHuang-07, BNS-85, BiNe-91, Bo-95,  Lo-97, Ma-85}). This model applies to contexts such as modeling an isothermal autocatalytic chemical reaction (where $n$ and $b$ represent the concentrations of the autocatalysts) or modeling thermal diffusive combustion (where $n$ is the concentration of the reactant and $b$ is the temperature of the mixture). Various results regarding wavefronts have been obtained depending on the value of $\lambda$ and the properties of $f$.  If $\lambda = 0$ or $\lambda > 0$ and $f(n,b) = nb$, an admissible closed half-line of wave speeds has been demonstrated (see~\cite{BiNe-91, Lo-97}). Conversely, if $\lambda > 0$ and $f(n,b) = \tilde{g}(n)b$ with $\tilde{g}(0) = 0$ and $\tilde{g}(s) > 0$ for all $s \in (0, 1]$, then an admissible closed half-line of wave speeds is guaranteed for $0 < \lambda < 1$ (see~\cite{Ma-85}). For $\lambda > 1$, however, the set of admissible speeds may be disconnected (see~\cite{Bo-95}). These results have been further improved in~\cite{AiHuang-07} to include more general reaction terms $f(n,b)$.

Without loss of generality, in the following, we assume in~\eqref{eq-sys} that 
\[D_n=1 \quad\text{and}\quad D_b=D>0.\]
When introducing the wave coordinate $\xi:=x-ct$, the functions $\eta$ and $\beta$ in~\eqref{wf} satisfy the system
\begin{subnumcases}{\label{eq-sys-1}}
	\eta''+c\eta'-\eta\beta=0, \label{eq-n} \\
	\left(D \eta\beta\beta'\right)'+c\beta'+\eta\beta=0, \label{eq-b}
\end{subnumcases}
where $'=\frac{\mathrm{d}}{\mathrm{d}\xi}$. The nontrivial stationary solutions associated with~\eqref{eq-sys-1} belong to one of the following sets:
\[
\{(\eta,\beta) \colon \eta \equiv K_\eta, \, \beta \equiv 0, \text{ with } K_\eta \in (0,+\infty)\}, 	\quad
\{(\eta,\beta) \colon \eta \equiv 0, \, \beta \equiv K_\beta, \text{ with } K_\beta \in (0,+\infty)\}.
\]
Here, we look for traveling waves of~\eqref{eq-sys} that connect nontrivial stationary solutions characterized by profiles $(\eta,\beta)$ consisting of a pair of strictly monotone functions (i.e., \textit{wavefronts}). In the following, we will consider $K_\eta=1=K_\beta$ and look for wavefronts that connect the two stationary solutions $(0,1)$ and $(1,0)$. This leads to the boundary conditions:
\begin{subnumcases}{\label{eq-bc}}
	\left(\eta(-\infty),\beta(-\infty)\right)=(0,1), \label{eq-bc1} \\
	\left(\eta(+\infty),\beta(+\infty)\right)=(1,0), \label{eq-bc2}
\end{subnumcases}
where for any function $m(\xi)$ with a limit at $\pm\infty$, we abbreviate $\lim_{\xi\to\pm\infty}m(\xi)=m(\pm\infty)$.

Given the presence of a degenerate diffusion term in system~\eqref{eq-sys}, before stating our results, let us point out the solution notion we are seeking (see~\cite{GiKe-04}).

\begin{definition}[Wavefront and semi-wavefront]\label{def-wave}
Let $J \subseteq \mathbb{R}$ be an open interval and $c \in \mathbb{R}$. Consider two non-constant monotone functions $\eta, \beta\colon J \to \mathbb{R}$ such that
\begin{enumerate}[nosep,wide=0pt, labelwidth=15pt, align=left]
\item[$(i)$] $\eta$ is of class $C^2$ and solves~\eqref{eq-n};
\item[$(ii)$] $\beta$ is continuous and differentiable a.e. with $\beta\beta'\in L^1_{loc}(J)$ and solves weakly~\eqref{eq-b}, namely, for every $\psi\in C^{\infty}_0(J)$, it satisfies
\begin{equation}\label{eq-weak}
\int_{J} \left[\left(D\eta(\xi)\beta(\xi)\beta'(\xi)+c\beta(\xi)\right)\psi'(\xi) -\eta(\xi)\beta(\xi)\psi(\xi) \right]\,\mathrm{d}\xi=0.
\end{equation}
\end{enumerate}
\begin{itemize}[nosep,wide=0pt, align=left]
\item When $J=\mathbb{R}$ and $(\eta,\beta)$ satisfies~\eqref{eq-bc}, $(n,b)$ with $n(x,t) = \eta(x-ct) = \eta(\xi)$ and $b(x,t) = \beta(x-ct) = \beta(\xi)$ is called (global) \emph{wavefront} of~\eqref{eq-sys}.
\item When $J=(-\infty,a)$ (or $J=(a,+\infty)$) with $a\in\mathbb{R}$ and $(\eta,\beta)$ satisfies~\eqref{eq-bc1} (or~\eqref{eq-bc2}), $(n,b)$ with $n(x,t) = \eta(x-ct) = \eta(\xi)$ and $b(x,t) = \beta(x-ct) = \beta(\xi)$ is called \emph{semi-wavefront} in $J$ of~\eqref{eq-sys}.
\end{itemize}

In both cases, the pair $(\eta,\beta)$ denotes the wave profile of $(n,b)$, whereas $c$ stands for the wave speed.
\end{definition}

\noindent\textit{Warning.} For conciseness, we will refer also to the profile $(\eta,\beta)$ as ``wavefront'' of~\eqref{eq-sys} or ``semi-wavefront'' of~\eqref{eq-sys} (satisfying either~\eqref{eq-bc1} or~\eqref{eq-bc2}).

\begin{definition}[Classical and sharp wavefront]\label{def-2}
A wavefront $(\eta,\beta)$ is said to be classical if the function $\beta$ is differentiable, $\beta\beta'$ is absolutely continuous, and the equation~\eqref{eq-b} holds~a.e.
On the other hand, a wavefront $(\eta,\beta)$ is said to be sharp at the point $\ell\in\{0,1\}$ if there exists a real number $\xi_\ell$ such that $\beta(\xi_\ell)=\ell$, and the function $\beta$ is classical on the set $\mathbb{R}\setminus\{\xi_\ell\}$, but not differentiable at the point $\xi_\ell$.
\end{definition}

This work investigates conditions under which wavefronts exist and examines their regularity properties. Specifically, we demonstrate the existence and uniqueness of a semi-wavefront for each positive wave speed in the negative half-line $(-\infty,0]$. Our approach utilizes shooting method and fixed-point theory in Fréchet spaces to establish existence and the central manifold theorem to ascertain uniqueness.

\begin{theorem}\label{th-1}
Let $D_n=1$ and $D_b=D>0$. For each $c>0$, system~\eqref{eq-sys} has a unique semi-wavefront in $(-\infty, 0]$ with wave speed $c$ whose profile $(\eta,\beta)$ satisfies~\eqref{eq-bc1}. Moreover, $\eta'>0$ and $\beta'<0$.
\end{theorem}

Further, we extend the solution to the positive half-line $[0, +\infty)$ and show that sufficiently large wave speeds yield a classical global wavefront.

\begin{theorem}\label{th-2}
Let $D_n=1$ and $D_b=D>0$. There exists $c_0>0$ such that, for every $c\geq c_0$, system~\eqref{eq-sys} has a  unique wavefront with wave speed $c$ whose profile $(\eta,\beta)$ satisfies~\eqref{eq-bc}.  If $c>c_0$, the wavefront is classical.
Moreover, the following estimates on $c_0$ hold:
\[
\max\left\{0,\sqrt{{D}/{15}}-1\right\}< c_0 \leq 2\sqrt{D e^D}.\]
\end{theorem}

Drawing parallels with the scenario when $D=0$, we conjecture the existence of a threshold wave speed for which a sharp wavefront for system~\eqref{eq-sys} can be observed (see~\cite{MHST-24, SMGA-01}).

\medskip
The paper is organized as follows. Section~\ref{section-3} presents preliminary properties of systems~\eqref{eq-sys-1}--\eqref{eq-bc} that are essential for further analysis. This section also discusses further formulation of the system provided by equation~\eqref{eq-bis}, which is ensured by the boundary conditions and the coupling of the reaction term. The proofs of Theorems~\ref{th-1} and~\ref{th-2} are detailed in Sections~\ref{section-4} and~\ref{section-5}, respectively.

\section{Preliminary proprieties}\label{section-3}

We first describe some properties that wavefront or semi-wavefronts of~\eqref{eq-sys} must have. To simplify expressions, the following notation is introduced:
\begin{equation}\label{eq-0}
\eta_0:=\eta(0),\quad \eta'_0:=\eta'(0), \qquad \beta_0:=\beta(0), \quad \beta'_0:=\beta'(0).
\end{equation}
Let $(\eta,\beta)$ be a solution of~\eqref{eq-sys-1}. It is straightforward to demonstrate that, for every $\xi_0 \in \mathbb{R}$, the shifted pair $\left(\eta(\cdot+\xi_0),\beta(\cdot+\xi_0)\right)$ remains a solution of~\eqref{eq-sys-1}.
So, with no loss of generality, in the following, we always assume that
\[
\beta_0>0.
\]
Furthermore, if $(\eta,\beta)$ is a wavefront of~\eqref{eq-sys} or a semi-wavefront of~\eqref{eq-sys} (satisfying either~\eqref{eq-bc1} or~\eqref{eq-bc2}), then according to Definition~\ref{def-wave}, the function $\beta$ is non-constant and monotone decreasing. If, moreover, $0$ belongs to the domain $J$ of $\beta$, then the following quantity
\begin{equation}\label{eq-tau}
\tau:=\sup\{\xi>0\colon\beta(\xi)>0\}.
\end{equation}
is straightforwardly well-defined. It is also easy to verify that $\tau\in\mathbb{R}\cup\{+\infty\}$.

The following lemma provides qualitative properties on wavefronts of~\eqref{eq-sys}.

\begin{lemma}\label{lem-1}
Every wavefront $(\eta,\beta)$ of~\eqref{eq-sys} with wave speed $c$ satisfies:
\begin{enumerate}[nosep,parsep=2pt,wide=0pt,  labelwidth=20pt, align=left]
\item[$(i)$] $\eta'(-\infty)=0$ and $\eta(-\infty)\beta'(-\infty)=0$;
\item[$(ii)$] $\eta'(+\infty)=0$ and $\beta(\tau^-)\beta'(\tau^-)=0$;
\item[$(iii)$] $c=\displaystyle\int_{-\infty}^{\tau}\eta(s)\beta(s)\, \mathrm{d}s>0$;
\item[$(iv)$] $0<\eta'(\xi)<c$ and $0<\eta(\xi)<1$, $\forall\xi\in\mathbb{R}$;
\item[$(v)$] $\beta'(\xi)<0$ and $0<\beta(\xi)<1$, $\forall\xi\in(-\infty,\tau)$.
\end{enumerate}
\end{lemma}

\begin{proof}
By Definition~\ref{def-wave}, we notice that $\eta$ and $\beta$ are monotone functions. Thus, they are bounded functions thanks to the boundary conditions~\eqref{eq-bc}, namely $\eta(\xi)\in[0,1]$ and $\beta(\xi)\in[0,1]$, for every $\xi\in\mathbb{R}$. By integrating~\eqref{eq-n} in $[\xi,0]$
with $\xi<0$, we have
\begin{equation}\label{eq:relation1}
\eta'(\xi)=\eta'_0+c\left(\eta_0-\eta(\xi)\right)-\int^0_\xi \eta(s)\beta(s)\, \mathrm{d}s,
\end{equation}
and so we deduce that $\eta'$ has limit as $\xi\to-\infty$. Since $\eta$ is bounded, then $\eta'(-\infty)=0$. Similarly, by integrating~\eqref{eq-b} in $[\xi,0]$ with $\xi<0$, we have
\[
D\eta(\xi)\beta(\xi)\beta'(\xi)=D\eta_0\beta_0\beta'_0+c\left(\beta_0-\beta(\xi)\right)+\int^0_\xi \eta(s)\beta(s)\, \mathrm{d}s,
\]
and so we deduce that $\eta(\xi)\beta(\xi)\beta'(\xi)$ has limit as $\xi\to-\infty$. Hence, $\eta(\xi)\beta'(\xi)$ has limit as $\xi\to-\infty$ because $\beta(-\infty)=1$. If, by contradiction $\eta(-\infty)\beta'(-\infty)=l<0$, then
\[
\eta(-\infty)\frac{\beta'(-\infty)}{\eta(-\infty)}=\frac{l}{\eta(-\infty)}=-\infty,
\]
which is a contradiction since $\beta$ is bounded. Thus, we have $\eta(-\infty)\beta'(-\infty)=0$ and statement $(i)$ is proved. 

The proof that $\eta'(+\infty)=0$ is analogous and also that $\beta(+\infty)\beta'(+\infty)=0$ when $\tau=+\infty$. Instead, when $\tau<+\infty$, we take $\psi \in C_0^{\infty}(\tau-\varepsilon, \tau+\varepsilon)$ for some $\varepsilon >0$ with $\psi(\tau) \ne 0$ and $0<\delta<\varepsilon$. By Definition~\ref{def-wave} we have that
\[\begin{split}
0=&\int_{\tau-\varepsilon}^{\tau+\varepsilon} \left[(D\eta(\xi)\beta(\xi)\beta'(\xi)+c\beta(\xi))\psi'(\xi)-\eta(\xi)\beta(\xi)\psi(\xi) \right]\, \mathrm{d}\xi\\
=&\displaystyle \lim_{\delta \to 0^+} \left[(D\eta(\xi)\beta(\xi)\beta'(\xi)+c\beta(\xi))\psi(\xi)  \right]_{\tau-\varepsilon}^{\tau-\delta}\\
& -\int_{\tau-\varepsilon}^{\tau-\delta} \left[\left( D\eta(\xi)\beta(\xi)\beta'(\xi) \right)' +c\beta'(\xi)+\eta(\xi)\beta(\xi) \right]\psi(\xi)\, \mathrm{d}\xi.
\end{split}\]
Since $\psi(\tau-\varepsilon)=0$ we get
\[
\lim_{\delta \to 0^+}\left[(D\eta(\tau-\delta)\beta(\tau-\delta)\beta'(\tau-\delta)+c\beta(\tau-\delta))\right]\psi(\tau-\delta)=0,
\]
implying
\[
\lim_{\xi \to \tau^- }\beta(\xi)\beta'(\xi)=0.
\]
This proves~$(ii)$.

To prove $(iii)$, let us integrate \eqref{eq-n} in $[\xi_0, \xi]$ and pass to the limit when $\xi_0 \to -\infty$; according to~$(i)$  we obtain that
\begin{equation}\label{e:etaprime}
\eta'(\xi)+c\eta(\xi)=\int_{-\infty}^{\xi}\eta(s)\beta(s) \, \mathrm{d}s.
\end{equation}
Hence, by~$(ii)$ and~\eqref{eq-bc2}, when passing to the limit in~\eqref{e:etaprime} as $\xi \to +\infty$, we obtain
\[
c=\int_{-\infty}^{+\infty}\eta(s)\beta(s) \, \mathrm{d}s=\int_{-\infty}^{\tau}\eta(s)\beta(s) \, \mathrm{d}s,
\]
proving~$(iii)$.

To prove the first property in $(iv)$, we observe that
\[\left(\eta''(\xi)+c\eta'(\xi)-\eta(\xi)\beta(\xi)\right)e^{c\xi}= \left(\eta'(\xi)e^{c\xi}\right)'-\eta(\xi)\beta(\xi)e^{c\xi}=0,\quad \forall\xi\in\mathbb{R}.
\]
Then, thanks to~$(i)$, we have
\begin{equation}\label{e:etader}
\eta'(\xi)=e^{-c\xi}\int_{-\infty}^{\xi}\eta(s)\beta(s)e^{cs}\,\mathrm{d}s.
\end{equation}
The case $\eta(\xi)\equiv 0$ in $(-\infty, \xi)$ is not possible because it implies $\eta=0$, in contradiction with Definition~\ref{def-wave}. Hence  we obtain that $\eta'(\xi)>0$, for every $\xi\in\mathbb{R}$. Now we prove that $0<\eta(\xi)<1$, for every $\xi\in\mathbb{R}$. Let us suppose by contradiction that there is $\xi_1\in\mathbb{R}$ such that $\eta(\xi_1)=0$. Since $\eta(\xi)\geq0$, for every $\xi\in\mathbb{R}$, then $\eta'(\xi_1)=0$, which is a contradiction and so $\eta(\xi)>0$, for every $\xi\in\mathbb{R}$. In a similar way, we can exclude the existence  of $\xi_2\in\mathbb{R}$ such that $\eta(\xi_2)=1.$ At last by~\eqref{e:etaprime} the function $\eta'+c\eta$ is increasing in $\mathbb{R}$ and so, by~$(ii)$ and~\eqref{eq-bc1}, we obtain $\eta'(+\infty)+c\eta(+\infty)=c.$ This, using~$(i)$, implies $\eta'(\xi)\leq c(1-\eta(\xi))<c$ for all $\xi\in \mathbb{R}$. Accordingly, $(iv)$ is  proved.

To prove the first property in $(v)$, we notice that $\beta$ is monotone decreasing due to Definition~\ref{def-wave} and boundary conditions~\eqref{eq-bc}. Thus, let us suppose by contradiction that $\beta'(\xi_3)=0$, for some $\xi_3<\tau$. By introducing the function $\Lambda\colon\mathbb{R}\to\mathbb{R}$ defined as
\begin{equation}\label{eq-lambda}
\Lambda(\xi)=D\eta(\xi)\beta(\xi)\beta'(\xi),
\end{equation}
we have $\Lambda(\xi_3)=0$ and $\Lambda'(\xi_3)<0$, thanks to~$(iv)$. This way, we obtain that $\Lambda$ is a strictly decreasing function in a neighborhood of $\xi_3.$ Moreover, $\mathrm{sgn}\left(\Lambda(\xi)\right)=\mathrm{sgn}\left(\beta'(\xi)\right)$ in a neighbourhood of $\xi_3.$ Then, we deduce that $\beta'>0$ in a left-neighbourhood of $\xi_3$, which is in contradiction with the monotonicity of $\beta$. We are left to prove that $\beta(\xi)<1$, for every $\xi<\tau.$ Thus, by contradiction, we assume that there is $\xi_4<\tau$ such that $\beta(\xi_4)=1$; then it follows $\beta(\xi)=1$, for every $\xi\leq\xi_4$. Consider $\xi_5 < \xi_4$, then $\beta(\xi) = 1$ for all $\xi \in (\xi_5 - \varepsilon, \xi_5 + \varepsilon)$ with $\varepsilon > 0$. According to Definition~\ref{def-wave}, $\beta$ satisfies~\eqref{eq-weak}. Thus, given $\psi \in C^\infty_0(\xi_5 - \varepsilon, \xi_5 + \varepsilon)$,  without loss of generality, we assume that $\psi<0$. Using~$(iv)$, we have
\[
0 = \int_{\xi_5 - \varepsilon}^{\xi_5 + \varepsilon} \left[D \eta(\xi) \beta(\xi) \beta'(\xi) + c \beta(\xi) \psi'(\xi) - \eta(\xi) \beta(\xi) \psi(\xi)\right] \, d\xi = \int_{\xi_5 - \varepsilon}^{\xi_5 + \varepsilon} -\eta(\xi) \psi(\xi) \, d\xi>0,
\]
which is a contradiction. Hence, $0<\beta(\xi)<1,$ for all $\xi \in (-\infty, \tau)$ and the proof is concluded.
\end{proof}

\begin{remark} \label{r:semi-w}
Let $(\eta, \beta)$ be a semi-wavefront of~\eqref{eq-sys} in $(-\infty, \tau)$ for some $c>0$. It is evident from the proof that it fulfills the properties stated in Lemma~\ref{lem-1}~$(i)$. Furthermore, for all $\xi<\tau$, we have $\eta(\xi)>0$, $\eta'(\xi)>0$, $0<\beta(\xi)<1$, and $\beta'(\xi)<0$, and when $\tau$ is finite, $\eta'(\tau)>0$.
Additionally, in Remark~\ref{rem-yy}, we will demonstrate that $\beta'(\tau^-) < 0$ is always satisfied when $\tau$ is finite.
\end{remark}

The following lemma highlights an essential property of wavefronts of~\eqref{eq-sys} or semi-wavefronts of~\eqref{eq-sys} on $(-\infty, \tau)$, which plays a significant role in the subsequent analysis.

\begin{lemma}\label{lemma-2}
If $(\eta, \beta)$ is a semi-wavefront of~\eqref{eq-sys} in $(-\infty, \tau)$ for some $c>0$,
then $\beta \in C^1(-\infty, \tau)$ and
\begin{equation}\label{eq-bis}
D\eta(\xi)\beta(\xi)\beta'(\xi)+c(\beta(\xi)-1)+\eta'(\xi)+c\eta(\xi)=0, \quad \forall \xi\in(-\infty,\tau).
\end{equation}
Conversely, if $(\eta,\beta)$ is a solution of \eqref{eq-n}--\eqref{eq-bis} in $(-\infty,\tau)$ satisfying~\eqref{eq-bc1} for some $c>0$, then $(\eta, \beta)$ is a semi-wavefront of~\eqref{eq-sys} in $(-\infty, \tau)$.
\end{lemma}

\begin{proof}
By Definition~\ref{def-wave}, if $(\eta, \beta)$ is a semi-wavefront of~\eqref{eq-sys} on $(-\infty, \tau)$, then it follows $D(\eta(\xi)\beta(\xi)\beta'(\xi))'=-c\beta'(\xi)-\eta''(\xi)-c\eta'(\xi) \in L^1_{loc}$. Thus the function $\xi \mapsto D\eta(\xi)\beta(\xi)\beta'(\xi)$ is continuous in $(-\infty, \tau).$ By Remark~\ref{r:semi-w}, since $\beta(\xi)>\beta(\tau)=0$, this implies that $\beta'$ is continuous in $(-\infty, \tau)$. Now, by taking into account~\eqref{eq-sys-1}, we have
\[
\eta''(\xi)+c\eta'(\xi)=\eta(\xi)\beta(\xi)=-\left(D\eta(\xi)\beta(\xi)\beta'(\xi)\right)'-c\beta'(\xi),\quad  \forall \xi\in(-\infty,\tau).
\]
By integrating in $[\xi_0,\xi]$ with $-\infty<\xi_0<\xi\leq\tau$, we obtain
\[
\int^{\xi}_{\xi_0}\left[\left(D \eta(s)\beta(s)\beta'(s)\right)'+c\beta'(s)+\eta''(s)+c\eta'(s)\right]\,\mathrm{d} s=0.
\]
By passing to the limit as $\xi_0\to-\infty$, the thesis follows thanks to~\eqref{eq-bc1}, Lemma~\ref{lem-1}~$(i)$ and Remark~\ref{r:semi-w}.
Differentiating~\eqref{eq-bis} and using~\eqref{eq-n}, the second statement of the lemma follows directly.
\end{proof}

\begin{remark}\label{r:betaC2}
Every $(\eta, \beta)$ semi-wavefront of~\eqref{eq-sys} in $(-\infty, \tau)$ for some $c>0$ satisfies $\beta \in C^2(-\infty, \tau)$.
In fact, in the proof of Lemma~\ref{lemma-2}, we showed that $\beta'$ is continuous in $(-\infty, \tau)$. By~\eqref{eq-bis} and the positivity of $\eta$ (see Remark~\ref{r:semi-w}), we have that
\[
\beta'(\xi)=\frac{c(1-\beta(\xi))}{D\eta(\xi)\beta(\xi)}-\frac{\eta'(\xi)}{D\eta(\xi)\beta(\xi)}-c\frac{1}{D\beta(\xi)}, \quad \forall\xi<\tau.
\]
At last, the conclusion follows thanks to the regularity of $\eta'$ (see Definition~\ref{def-wave}).
\end{remark}

The following lemma contains some further important properties satisfied by the semi-wavefronts $(\eta,\beta)$ of~\eqref{eq-sys} in $(-\infty, \tau)$ with $c>0$.

\begin{lemma}\label{l:limits}
If $(\eta, \beta)$ is a wavefront of~\eqref{eq-sys} in $(-\infty, \tau)$ for some $c>0$,
then
\begin{enumerate}[nosep,parsep=2pt,wide=0pt,  labelwidth=20pt, align=left]
\setlength{\baselineskip}{25pt}
\item[$(i)$] $\displaystyle \lim_{\xi \to -\infty}\frac{\eta'(\xi)}{\eta(\xi)}=\frac{2}{c+\sqrt{c^2+4}}$;
\item[$(ii)$] $\beta'(-\infty)=0$;
\item[$(iii)$] $\displaystyle \lim_{\xi \to -\infty}\frac{1-\beta(\xi)}{\eta(\xi)}=\frac{2}{c^2+c\sqrt{c^2+4}}+1$.
\end{enumerate}
\end{lemma}
\begin{proof}
To prove~$(i)$, we assume, by a contradiction, that such a limit does not exist. Therefore,  by using Lemma~\ref{lem-1}~$(iv)$, we have
\begin{equation}\label{e:contr1}
0\leq \ell_1:=\liminf_{\xi \to -\infty }\frac{\eta'(\xi)}{\eta(\xi)}<\limsup_{\xi \to -\infty }\frac{\eta'(\xi)}{\eta(\xi)}=:\ell_2\leq +\infty.
\end{equation}
For every $\lambda \in (\ell_1, \ell_2)$ there exist two sequences $(\xi_n)_n$ and $(\mu_n)_n$ such that $\xi_n\to -\infty,$ $\mu_n \to -\infty$ as $n \to +\infty$, and the following hold
\begin{align}
&\frac{\eta'(\xi_n)}{\eta(\xi_n)}=\frac{\eta'(\mu_n)}{\eta(\mu_n)}=\lambda, \quad \forall n\in\mathbb{N},\nonumber\\
&\left(\frac{\eta'(\xi)}{\eta(\xi)}\right)'\Big\vert_{\xi=\xi_n}>0, \quad \forall n\in\mathbb{N},\label{e:deriv-a}\\
&\left(\frac{\eta'(\xi)}{\eta(\xi)}\right)'\Big\vert_{\xi=\mu_n}<0, \quad\forall n\in\mathbb{N}.\label{e:deriv-b}
\end{align}
When computing the derivative in \eqref{e:deriv-a} we have that $\eta''(\xi_n)/\eta(\xi_n)>\lambda^2 $ for all $n$ and then, by  ~\eqref{eq-n},  $-c\lambda+\beta(\xi_n)>\lambda^2$; passing to the limit as $n \to +\infty$ we obtain that $\lambda^2+c\lambda-1\leq 0$. Since $\lambda$ is positive by~\eqref{e:contr1}, this implies
\[
0< \lambda \leq \frac{2}{c+\sqrt{c^2+4}}.
\]
Similarly, when considering \eqref{e:deriv-b}, we obtain
\[
\lambda \ge \frac{2}{c+\sqrt{c^2+4}}.
\]
It follows that $\displaystyle \lambda = \frac{2}{c+\sqrt{c^2+4}}$. This is in contradiction with $\lambda$ arbitrarily taken in $(\ell_1, \ell_2)$, so we have proved that the limit exists. Now, let
\[
\ell:=\lim_{\xi \to -\infty}\frac{\eta'(\xi)}{\eta(\xi)}.
\]
By Lemma~\ref{lem-1}~$(iv)$, we notice that $0\leq \ell\leq +\infty$. Due to
l'H\^ospital role, we have that
\[
\lim_{\xi \to -\infty}\frac{1}{\eta(\xi)}\int_{-\infty}^{\xi}\eta(s)\beta(s)\,\mathrm{d}s=\lim_{\xi \to -\infty}\frac{\eta(\xi)\beta(\xi)}{\eta'(\xi)}=\frac{1}{\ell}
\]
with the obvious meaning if $\ell=0^+$ or $\ell=+\infty$. By \eqref{e:etaprime}
\[
\ell=\lim_{\xi \to -\infty}\frac{\eta'(\xi)}{\eta(\xi)}=\lim_{\xi \to -\infty}-c+\frac{\int_{-\infty}^{\xi}\eta(s)\beta(s)\,\mathrm{d}s}{\eta(\xi)}=-c+\frac{1}{\ell}.
\]
This proves that $\ell \in (0, +\infty)$ and $\ell^2+c\ell-1=0$. Condition~$(i)$ is then satisfied.

To prove~$(ii)$, we assume, by a contradiction, that such a limit does not exist and assume the existence of $b_1, b_2$ satisfying
\[
b_1:=\liminf_{\xi \to -\infty }\beta'(\xi)<\limsup_{\xi \to -\infty }\beta'(\xi)=:b_2.
\]
By Remark \ref{r:semi-w}, we notice that $\beta'(\xi)<0$ for all $\xi\in(-\infty, \tau)$. This implies $-\infty\leq b_1<b_2\leq 0$. In addition $b_2=0$ since, otherwise, we could find $\varepsilon >0$ and $\xi_0<0$ such that $\beta'(\xi)<-\varepsilon,$ for $\xi \leq \xi_0$ in contradiction with the boundedness of $\beta$. Due to~\eqref{eq-bc1}, we have that
\[
Db_1=\liminf_{\xi \to -\infty }D\beta(\xi)\beta'(\xi)<\limsup_{\xi \to -\infty }D\beta(\xi)\beta'(\xi)=0.
\]
Hence, for every  $\lambda \in (Db_1, 0)$, there exists a sequence $(\xi_n)_n\subset (-\infty, \tau)$ with $\xi_n \to -\infty$ as $n \to +\infty$ such that
\[
D\beta(\xi_n)\beta'(\xi_n)=\lambda, \qquad \text{and} \qquad \left(D\beta(\xi)\beta'(\xi)  \right)'\Big\vert_{\xi=\xi_n}<0, \qquad \forall n.
\]
By writing~\eqref{eq-bis} as
\begin{equation}\label{e:bbis2}
D\beta(\xi)\beta'(\xi)=\frac{c(1-\beta(\xi))}{\eta(\xi)}-\frac{\eta'(\xi)}{\eta(\xi)}-c, \quad \xi\in(-\infty, \tau),
\end{equation}
we have, in the whole interval $(-\infty, \tau)$,
\[
\left(D\beta \beta'\right)'=-\frac{c\beta'}{\eta}-\frac{c(1-\beta)}{\eta}\frac{\eta'}{\eta}-\frac{\eta''}{\eta}+\left( \frac{\eta'}{\eta} \right)^2.
\]
According to \eqref{eq-n}, in $(-\infty, \tau)$, we obtain
\begin{equation}\label{e:bbisderiv}
\left(D\beta \beta'\right)'=-\frac{c\beta'}{\eta}-\frac{c(1-\beta)}{\eta}\frac{\eta'}{\eta}+\frac{c\eta'}{\eta}-\beta+\left( \frac{\eta'}{\eta} \right)^2.
\end{equation}
Moreover, if we restrict to the values of $(\xi_n)_n,$ thanks to \eqref{e:bbis2}, we have
\[
\beta'(\xi_n)=\frac{\lambda}{D\beta(\xi_n)}, \qquad \frac{c(1-\beta(\xi_n))}{\eta(\xi_n)}=\lambda+\frac{\eta'(\xi_n)}{\eta(\xi_n)}+c.
\]
Therefore, by \eqref{e:bbisderiv}, it follows
\[\begin{split}
0&>\left(D\beta(\xi) \beta'(\xi)\right)'\Big\vert_{\xi=\xi_n}\\
&=-\frac{c\lambda}{D\eta(\xi_n)\beta(\xi_n)}-\left[ \lambda + \frac{\eta'(\xi_n)}{\eta(\xi_n)}+c\right]\frac{\eta'(\xi_n)}{\eta(\xi_n)} +\frac{c\eta'(\xi_n)}{\eta(\xi_n)}-\beta(\xi_n)+\left( \frac{\eta'(\xi_n)}{\eta(\xi_n)}\right)^2.
\end{split}\]
Due to property~$(i)$ and~\eqref{eq-bc1}, when passing to the limit as $n \to +\infty$, we have that
\[
0\ge \lim_{n \to +\infty}D(\beta(\xi) \beta'(\xi))'|_{\xi=\xi_n}=+\infty,
\]
which is a contradiction. Thus, the limit $\beta'(-\infty)$ exists and is necessarily $0$, so~$(ii)$ is proved.

Property~$(iii)$ follows as an immediate consequence of~$(i)$ and~$(ii)$, thanks to~\eqref{e:bbis2}.
\end{proof}

\begin{proposition}\label{prop-bound}
Every wavefront $(\eta,\beta)$ of~\eqref{eq-sys} with wave speed $c$ satisfies:
\begin{equation}\label{eq-nb}
\eta(\xi)\geq L(c) \left(1-\beta(\xi)\right),\quad \forall \xi\in\mathbb{R},
\end{equation}
where $\displaystyle L(c):=\frac{c^2+c\sqrt{c^2+4}}{2+c^2+c\sqrt{c^2+4}}$.
\end{proposition}

\begin{proof}
Firstly, we claim that
\begin{equation}\label{eq-nnp}
\frac{\eta'(\xi)}{\eta(\xi)}\leq \frac{2}{c+\sqrt{c^2+4}}, \quad \forall \xi\in\mathbb{R}.
\end{equation}
To prove this assertion, we observe that the function ${\eta'}/{\eta}$ is always positive. Next, we find its critical points and, by referring to equation~\eqref{eq-n}, we deduce
\[
\left(\frac{\eta'}{\eta}\right)'(\xi)=0 \Leftrightarrow \left(\frac{\eta'(\xi)}{\eta(\xi)}\right)^2+c\frac{\eta'(\xi)}{\eta(\xi)}-\beta(\xi)=0.
\]
This condition is met when
\begin{equation}\label{eq-x1}
\frac{\eta'(\xi)}{\eta(\xi)}=\frac{2\beta(\xi)}{c+\sqrt{c^2+4\beta(\xi)}}.
\end{equation}
Given that the function
\[
\varphi(v)=\frac{2 v}{c+\sqrt{c^2+4v}}, \quad v\in[0,1],
\]
satisfies $\varphi'(v)>0$ for all $v\in[0,1]$, it follows that
\begin{equation}\label{eq-x2}
\frac{2\beta(\xi)}{c+\sqrt{c^2+4\beta(\xi)}}\leq \frac{2}{c+\sqrt{c^2+4}}, \quad \forall \xi\in\mathbb{R}.
\end{equation}
Employing Lemma~\ref{l:limits}~$(i)$ and noting that
\[\lim_{\xi\to+\infty}\frac{\eta'(\xi)}{\eta(\xi)}=0,\]
we conclude from \eqref{eq-x1} and \eqref{eq-x2} that the claim~\eqref{eq-nnp} is valid.

Applying equation~\eqref{eq-bis}, we derive that
\[
\frac{1-\beta(\xi)}{\eta(\xi)}=1+\frac{\eta'(\xi)}{c\eta(\xi)}+\frac{D\beta(\xi)\beta'(\xi)}{c}<1+\frac{\eta'(\xi)}{c\eta(\xi)}\leq 1+ \frac{2}{c^2+c\sqrt{c^2+4}},\quad \forall \xi\in\mathbb{R},
\]
which completes the proof.
\end{proof}

The following result provides a positive lower bound for wavefront speeds of the system~\eqref{eq-sys} under certain conditions on the diffusion coefficient $D$.

\begin{proposition}\label{est-c}
Every wavefront $(\eta,\beta)$ of~\eqref{eq-sys} with wave speed $c$ satisfies:
\[
c> \max\left\{0,\sqrt{{D}/{15}}-1\right\}.
\]
\end{proposition}

\begin{proof}
Firstly, we multiply~\eqref{eq-b} for $\beta$ and we integrate in $(\xi_0,\xi_1)\subset(-\infty, \tau)$, so that we obtain
\[\begin{split}
0=&\int_{\xi_0}^{\xi_1}\left(D\eta(s)\beta(s)\beta'(s)	\right)'\beta(s)\,\mathrm{d}s+c\int_{\xi_0}^{\xi_1}\beta(s)\beta'(s)\,\mathrm{d}s+\int_{\xi_0}^{\xi_1}\eta(s)\beta^2(s)\,\mathrm{d}s\\
=&D\eta(\xi_1)\beta^2(\xi_1)\beta'(\xi_1)-D\eta(\xi_0)\beta^2(\xi_0)\beta'(\xi_0)-\int_{\xi_0}^{\xi_1}D\eta(s)\beta(s)\beta'^2(s)\,\mathrm{d}s\\
&+c\left(\frac{\beta(\xi_1)^2}{2}-\frac{\beta(\xi_0)^2}{2}\right)+\int_{\xi_0}^{\xi_1}\eta(s)\beta^2(s)\,\mathrm{d}s.
\end{split}\]
By letting $\xi_0\to-\infty$ and $\xi_1\to \tau$, and using the boundary conditions~\eqref{eq-bc1} along with Lemma~\ref{lem-1}, we have
\begin{equation}\label{eq-est1}
-\int_{-\infty}^{\tau}D\eta(s)\beta(s)\beta'^2(s)\,\mathrm{d}s-\frac{c}{2}+\int_{-\infty}^{\tau}\eta(s)\beta^2(s)\,\mathrm{d}s=0.
\end{equation}
By using Lemma~\ref{lem-1}, we have
\[
\int_{-\infty}^{\tau}\eta(s)\beta^2(s)\,\mathrm{d}s\leq\int_{-\infty}^{\tau}\eta(s)\beta(s)\,\mathrm{d}s= c,
\]
and so from~\eqref{eq-est1} we deduce
\begin{equation}\label{eq-est21}
\int_{-\infty}^{\tau}D\eta(s)\beta(s)\beta'^2(s)\,\mathrm{d}s\leq \frac{c}{2}.
\end{equation}

Next, we multiply~\eqref{eq-b} for $D\eta\beta\beta'$ and we integrate in $(\xi_0,\xi_1)\subset(-\infty, +\infty)$, so that we obtain
\[\begin{split}
0=&\int_{\xi_0}^{\xi_1}\left(D\eta(s)\beta(s)\beta'(s)	\right)'D\eta(s)\beta(s)\beta'(s)\,\mathrm{d}s+c\int_{\xi_0}^{\xi_1}\beta'(s)D\eta(s)\beta(s)\beta'(s)\,\mathrm{d}s\\
&+\int_{\xi_0}^{\xi_1}D\beta'(s)\eta^2(s)\beta^2(s)\,\mathrm{d}s\\
=&\frac12\left(D\eta(\xi_1)\beta(\xi_1)\beta'(\xi_1)\right)^2-\frac12\left(D\eta(\xi_0)\beta(\xi_0)\beta'(\xi_0)\right)^2+c\int_{\xi_0}^{\xi_1}D\eta(s)\beta(s)\beta'^2(s)\,\mathrm{d}s\\
&+\int_{\xi_0}^{\xi_1}D\beta'(s)\eta^2(s)\beta^2(s)\,\mathrm{d}s.
\end{split}\]
By letting $\xi_0\to-\infty$ and $\xi_1\to \tau$, and using the boundary conditions~\eqref{eq-bc} along with Lemma~\ref{lem-1}, we have
\begin{equation}\label{eq-est2}
c\int_{-\infty}^{\tau}D\eta(s)\beta(s)\beta'^2(s)\,\mathrm{d}s+\int_{-\infty}^{\tau}D\beta'(s)\eta^2(s)\beta^2(s)\,\mathrm{d}s=0.
\end{equation}

From~\eqref{eq-est21} and~\eqref{eq-est2}, we deduce
\begin{equation}\label{eq-est5}
\frac{c^2}{2}\geq-\int_{-\infty}^{\tau}D\beta'(s)\eta^2(s)\beta^2(s)\,\mathrm{d}s.
\end{equation}
Thanks to Proposition~\ref{prop-bound} and applying inequality~\eqref{eq-nb}, we obtain
\[\begin{split}
\frac{c^2}{2}&\geq\int_{-\infty}^{\tau}D\left(-\beta'(s)\right)\eta^2(s)\beta^2(s)\,\mathrm{d}s\\
&\geq \int_{-\infty}^{\tau}D\left(-\beta'(s)\right)L(c)^2\left(1-\beta(s)\right)^2 \beta^2(s)\,\mathrm{d}s\\
&= D L^2(c)\int_{0}^{1}\left(1-y\right)^2y^2\,\mathrm{d}y= \frac{DL^2(c)}{30},\\
\end{split}\]
with $L(c)$ defined in Proposition~\ref{prop-bound}.
That leads us to conclude that the wave speed $c$ must satisfy the following condition:
\begin{equation}\label{aux-est}
\left(\frac{c}{L(c)}\right)^2=\left(\frac{2+c^2+c\sqrt{c^2+4}}{c+\sqrt{c^2+4}}\right)^2
\geq \frac{D}{15}.
\end{equation}
Given that the function $c\mapsto\left({c}/{L^2(c)}\right)$ is strictly increasing, it follows that for $0<D\leq15$, inequality~\eqref{aux-est} is satisfied for all $c\geq0;$ hence, $c>0$ by~Lemma~\ref{lem-1}~$(iii)$. Conversely, for $D>15$, the inequality~\eqref{aux-est} yields
\[
1+c>\frac{2}{c+\sqrt{c^2+4}}+c\geq \sqrt{\frac{D}{15}},
\]
and so $c> \sqrt{{D}/{15}}-1$. This completes the proof.
\end{proof}

Now, we provide a first-order reduction that will be useful in the following.
\begin{remark}[Reduction to a singular first-order problem]
Let $(\eta, \beta)$ be a wavefront of~\eqref{eq-sys} for some wave speed $c>0$.
By the strict monotonicity of $\beta$ in $(-\infty, \tau)$ (Lemma~\ref{lem-1}~$(v)$), we can define
\begin{equation}\label{e:Nfun}
N(\beta):=\eta(\xi(\beta)), \quad 0<\beta<1,
\end{equation}
where $\xi=\xi(\beta)$ is the inverse function of $\beta$ in $(-\infty, \tau)$.
By Lemma~\ref{lem-1}$(iv)$--$(v)$, we have
\[
\dot{N}(\beta)=\frac{\eta'(\xi)}{\beta'(\xi)}<0, \quad 0<\beta<1,
\]
where $\dot{\hphantom{N}} =\frac{\mathrm{d}}{\mathrm{d}\beta}$. Moreover, by using~\eqref{eq-bc1}, we obtain
\[
\lim_{\beta\to 0^+} N(\beta)=\lim_{\xi \to \tau}\eta(\xi)=\eta(\tau), \qquad \lim_{\beta\to 1^-} N(\beta)=\lim_{\xi \to -\infty}\eta(\xi)=0.
\]
Hence $N$ can be extended on the closed interval $[0,1]$ and by Lemma~\ref{l:limits}~$(iii)$, it follows
\[
\dot{N}(1)=\lim_{\beta \to 1^-}\frac{N(\beta)}{\beta -1}=\lim_{\xi \to -\infty}-\frac{\eta(\xi)}{1-\beta(\xi)}=- \frac{c^2+c\sqrt{c^2+4}}{2+c^2+c\sqrt{c^2+4}}.
\]
We set
\begin{equation}\label{e:z}
z(\beta):=DN(\beta)\beta\beta'(\xi(\beta)), \quad 0<\beta<1.
\end{equation}
Due to Lemma~\ref{lem-1}~$(v)$, we have $z(\beta)<0$, for all $0<\beta<1$. From~\eqref{eq-b}, we obtain $\dot z(\beta)=-c-{DN^2(\beta)\beta^2}/{z},$ for all $0<\beta<1.$ Moreover, by Lemma~\ref{lem-1}~$(i)$--$(ii)$, we have that $z(0)=z(1)=0$.
Therefore, to each wavefront of \eqref{eq-sys} corresponds the function $z(\beta)$ defined in~\eqref{e:z} satisfying the following first-order singular boundary value problem
\begin{subnumcases}{\label{e:zsist}}
\dot{z}(\beta)=-c-\frac{DN^2(\beta)\beta^2}{z(\beta)}, \quad 0<\beta<1, \label{e:z1} \\
z(\beta)<0, \quad 0<\beta<1, \label{e:z2}\\
z(0)=z(1)=0, \label{e:z3}
\end{subnumcases}
with $z \in C[0,1] \cap C^1(0,1)$.
\end{remark}

\begin{lemma}\label{l:beta'intau}
If $(\eta,\beta)$ is a wavefront of~\eqref{eq-sys}, then $\beta'(\tau^-)$ exists and satisfies one of the following conditions:
$\beta'(\tau^-)=0$ or $\beta'(\tau^-)=-\displaystyle \frac{c}{D\eta(\tau)}$.
\end{lemma}

\begin{proof}
Let $z(\beta)$ be the function defined as in \eqref{e:z}. By \cite[Proposition 8.2]{BCM1} the limit $\dot{z}(0^+)=\displaystyle \lim_{\beta \to 0^+} \frac{z(\beta)}{\beta}$ exists and it is either $\dot{z}(0^+)=-c$ or $\dot{z}(0^+)=0$. We have that
\[
\lim_{\xi\to\tau^-}\beta'(\xi)=\lim_{\beta \to 0^+}\beta'(\xi(\beta))=\lim_{\beta \to 0^+}\frac{1}{DN(\beta)}\frac{z(\beta)}{\beta}
\]
with $\xi(\beta)$ and $N(\beta)$ as in~\eqref{e:Nfun}. Since, $\lim_{\beta\to 0} N(\beta)=\eta(\tau)$, the claim is proved.
\end{proof}

\begin{remark}\label{rem-yy}
According to Lemma~\ref{lem-1}~$(v)$, the wavefronts $(\eta,\beta)$ of~\eqref{eq-sys} satisfy $\beta(\xi)<1$ for every $\xi\in\mathbb{R}$, which implies that wavefronts can only be sharp at $\ell=0$ (cf., Definition~\ref{def-2}). Depending on the value of $\tau$, as defined in~\eqref{eq-tau}, and Lemma~\ref{l:beta'intau}, wavefronts can be of two types: classical, if $\tau=+\infty$ or $\tau<+\infty$ with $\beta'(\tau^-)=0$, or sharp, if $\tau<+\infty$ with $\beta'(\tau^-)= -\frac{c}{D\eta(\tau)}$. In Proposition~\ref{p:tau}, we will show that the case $\tau<+\infty$ and $\beta'(\tau^-)=0$ can not occur. Figure~\ref{fig-1} illustrates the two types of admissible wavefronts of~\eqref{eq-sys}.
\end{remark}

\begin{figure}[htb]
\centering
\begin{tikzpicture}
\begin{axis}[legend style={at={(axis cs:72,30)},anchor=north west},
  tick label style={font=\scriptsize},
  axis y line=middle,
  axis x line=middle,
  ytick={0,1},
  yticklabel style={anchor=south east},
  xtick={0,7},
  xticklabels={$0$,$\tau$},
  extra y tick style={anchor=north east,yticklabel style={anchor=east,yshift=-0.5mm}},
  extra y ticks={0.22,0.75,0.87,1},
  extra y tick labels={$\eta_0$,$\eta_{\tau}$,$\beta_0$},
  xlabel={\small $\xi$}, ylabel={},
every axis x label/.style={
    at={(ticklabel* cs:1.0)},
    anchor=west,
},
every axis y label/.style={
    at={(ticklabel* cs:5.0)},
    anchor=south west
},
  set layers,
  width=8cm,
  height=4.5cm,
  xmin=-12,
  xmax=12,
  ymin=-0.1,
  ymax=1.5]
\addplot [draw=black, line width=0.6pt, smooth, on layer=axis background]coordinates {(-12,1)(12,1)};
\addplot [draw=black, dotted, line width=0.3pt, smooth, on layer=axis background]coordinates {(7,0)(7,0.75)};
\addplot [draw=black, dotted, line width=0.3pt, smooth, on layer=axis background]coordinates {(0,0.75)(7,0.75)};
\addplot [draw=black, line width=1pt, smooth]coordinates {(12,0) (7,0)};
\addplot [draw=black, line width=1pt, smooth]coordinates {(-12.0000,0.9700) (-11.3540,0.9704) (-10.7280,0.9707) (-10.1220,0.9707) (-9.5344,0.9706) (-8.9660,0.9703) (-8.4160,0.9698) (-7.8838,0.9691) (-7.3691,0.9681) (-6.8715,0.9668) (-6.3904,0.9653) (-5.9255,0.9635) (-5.4763,0.9613) (-5.0424,0.9588) (-4.6234,0.9560) (-4.2188,0.9528) (-3.8281,0.9492) (-3.4510,0.9452) (-3.0871,0.9408) (-2.7358,0.9360) (-2.3967,0.9307) (-2.0695,0.9250) (-1.7536,0.9188) (-1.4487,0.9120) (-1.1543,0.9048) (-0.8700,0.8970) (-0.5952,0.8887) (-0.3297,0.8798) (-0.0730,0.8703) (0.1754,0.8602) (0.4159,0.8495) (0.6490,0.8382) (0.8750,0.8263) (1.0944,0.8136) (1.3077,0.8003) (1.5153,0.7863) (1.7175,0.7715) (1.9149,0.7560) (2.1079,0.7398) (2.4824,0.7051) (2.6648,0.6865) (2.8445,0.6671) (3.1975,0.6258) (3.3718,0.6039) (3.5451,0.5810) (3.8906,0.5327) (4.0637,0.5071) (4.2376,0.4806) (4.5896,0.4246) (5.3223,0.3005) (5.5140,0.2668) (5.7099,0.2321) (6.1165,0.1593) (6.3279,0.1212) (6.5453,0.0820) (6.7692,0.0416) (7.0000,0.0000)};
\addplot [draw=black, line width=1pt, smooth]coordinates {(12.0000,0.9800) (11.6780,0.9793) (11.3690,0.9772) (11.0720,0.9738) (10.7860,0.9692) (10.5100,0.9633) (10.2440,0.9562) (9.9872,0.9479) (9.7387,0.9385) (9.4979,0.9281) (9.2641,0.9167) (9.0366,0.9043) (8.8148,0.8910) (8.5979,0.8768) (8.3853,0.8617) (7.9702,0.8293) (7.7662,0.8120) (7.5638,0.7940) (7.1607,0.7563) (6.3425,0.6749) (6.1314,0.6536) (5.9163,0.6319) (5.4715,0.5879) (4.5038,0.4981) (4.2415,0.4756) (3.9699,0.4531) (3.3953,0.4084) (3.0910,0.3863) (2.7745,0.3644) (2.1021,0.3216) (1.7447,0.3006) (1.3724,0.2801) (0.9844,0.2600) (0.5801,0.2404) (0.1587,0.2214) (-0.2804,0.2029) (-1.2145,0.1679) (-1.7110,0.1514) (-2.2279,0.1357) (-2.7659,0.1208) (-3.3259,0.1067) (-3.9084,0.0936) (-4.5141,0.0813) (-5.1437,0.0701) (-5.7979,0.0598) (-6.4774,0.0506) (-7.1829,0.0426) (-7.9151,0.0357) (-8.6746,0.0300) (-9.4622,0.0255) (-10.2780,0.0223) (-11.1240,0.0205) (-12.0000,0.0200)};
\node at (axis cs:10,0.2) {\footnotesize{$\beta$}};
\node at (axis cs:10,0.8) {\footnotesize{$\eta$}};
\end{axis}
\end{tikzpicture}
\quad
\begin{tikzpicture}
\begin{axis}[legend style={at={(axis cs:72,30)},anchor=north west},
  tick label style={font=\scriptsize},
  axis y line=middle,
  axis x line=middle,
  ytick={0,1},
  yticklabel style={anchor=south east},
  xtick={0,12},
  xtick style={draw=none},
  xticklabels={0,\color{white}$\tau$},
  extra y tick style={anchor=north east,yticklabel style={anchor=east,yshift=-0.5mm}},
  extra y ticks={0.31,0.71,1},
  extra y tick labels={$\eta_0$,$\beta_0$},
  xlabel={\small $\xi$}, ylabel={},
every axis x label/.style={
    at={(ticklabel* cs:1.0)},
    anchor=west,
},
every axis y label/.style={
    at={(ticklabel* cs:5.0)},
    anchor=south west
},
  set layers,
  width=8cm,
  height=4.5cm,
  xmin=-12,
  xmax=12,
  ymin=-0.1,
  ymax=1.5]
\addplot [draw=black, line width=0.6pt, smooth, on layer=axis background]coordinates {(-12,1)(12,1)};
\addplot [draw=black, line width=1pt, smooth]coordinates {(12.0000,0.9700) (11.6750,0.9689) (11.3580,0.9665) (11.0470,0.9628) (10.7420,0.9580) (10.4420,0.9520) (10.1480,0.9449) (9.8590,0.9368) (9.5742,0.9275) (9.2934,0.9174) (9.0163,0.9062) (8.7423,0.8941) (8.4712,0.8812) (8.2024,0.8675) (7.9356,0.8529) (7.4063,0.8216) (7.1429,0.8050) (6.8798,0.7877) (6.3530,0.7514) (5.2852,0.6732) (3.0000,0.5031) (2.6923,0.4814) (2.3784,0.4597) (1.7300,0.4165) (1.3947,0.3951) (1.0516,0.3740) (0.3398,0.3324) (-0.0295,0.3120) (-0.4085,0.2920) (-1.1968,0.2532) (-1.6069,0.2346) (-2.0283,0.2164) (-2.9063,0.1818) (-3.3637,0.1654) (-3.8340,0.1498) (-4.3176,0.1348) (-4.8149,0.1206) (-5.3263,0.1073) (-5.8522,0.0947) (-6.3931,0.0831) (-6.9492,0.0724) (-7.5211,0.0626) (-8.1092,0.0539) (-8.7138,0.0462) (-9.3354,0.0396) (-9.9745,0.0342) (-10.6310,0.0299) (-11.3060,0.0268) (-12.0000,0.0250)};
\addplot [draw=black, line width=1pt, smooth]coordinates {(-12.0000,0.9700) (-11.1740,0.9705) (-10.3820,0.9697) (-9.6226,0.9674) (-8.8960,0.9638) (-8.2007,0.9590) (-7.5357,0.9529) (-6.9001,0.9456) (-6.2930,0.9371) (-5.7133,0.9276) (-5.1601,0.9170) (-4.6325,0.9054) (-4.1294,0.8928) (-3.6500,0.8793) (-3.1932,0.8649) (-2.7581,0.8497) (-2.3438,0.8337) (-1.9492,0.8169) (-1.5734,0.7995) (-1.2155,0.7813) (-0.8745,0.7626) (-0.5494,0.7434) (-0.2393,0.7236) (0.3398,0.6826) (0.6108,0.6615) (0.8707,0.6401) (1.3608,0.5964) (1.5931,0.5743) (1.8181,0.5519) (2.2500,0.5070) (2.4589,0.4845) (2.6644,0.4620) (3.0688,0.4172) (3.2698,0.3950) (3.4711,0.3731) (3.8789,0.3299) (4.0873,0.3088) (4.2999,0.2881) (4.5177,0.2678) (4.7417,0.2480) (4.9728,0.2287) (5.2121,0.2100) (5.7188,0.1743) (5.9881,0.1576) (6.2693,0.1415) (6.5635,0.1263) (6.8716,0.1118) (7.1945,0.0983) (7.5331,0.0857) (7.8886,0.0740) (8.2617,0.0634) (8.6535,0.0538) (9.0650,0.0453) (9.4971,0.0379) (9.9507,0.0318) (10.4270,0.0269) (10.9260,0.0232) (11.4510,0.0209) (12.0000,0.0200)};
\node at (axis cs:10,0.2) {\footnotesize{$\beta$}};
\node at (axis cs:10,0.8) {\footnotesize{$\eta$}};
\end{axis}
\end{tikzpicture}
\caption{Admissible wavefronts for~\eqref{eq-sys} with profile $(\eta,\beta)$ satisfying~\eqref{eq-bc}. On the left sharp profile ($\tau<+\infty$), on the right classical profile ($\tau=+\infty$). }\label{fig-1}
\end{figure}
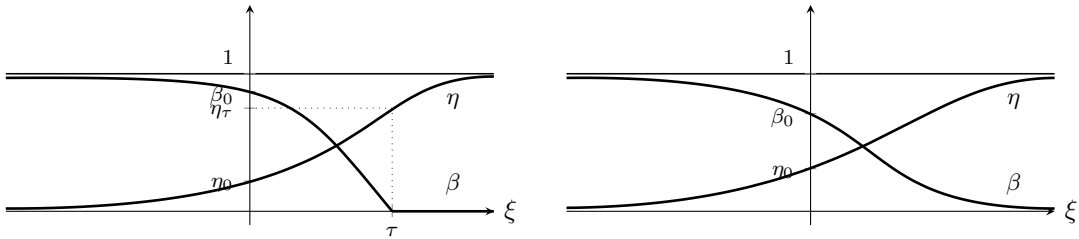

First we consider a semi-wavefront $(\eta, \beta)$ of~\eqref{eq-sys} in $(-\infty,\tau)$ and we provide an upper estimate for $\eta(\tau^-)$ which is independent on $c$.

\begin{proposition}\label{p:equilim}
If $(\eta, \beta)$ is a semi-wavefront of~\eqref{eq-sys} in $(-\infty, \tau)$ for some $c>0$, then
\begin{equation}\label{eq-etad}
\eta(\tau^-)\leq \sqrt{e^D}.
\end{equation}
\end{proposition}

\begin{proof}
If $\eta(\tau^-)\leq1$, then~\eqref{eq-etad} is trivially satisfied. Therefore, suppose that $\eta(\tau^-)>1$. Then, by~\eqref{eq-bc1} and since $\eta$ is strictly increasing in $(0,\tau)$ (see Lemma~\ref{lem-1} and Remark \ref{r:semi-w}), there exists $\xi_0 \in (0, \tau)$ such that $\eta(\xi_0)=1$ and $\eta(\xi)>1$ for $\xi \in (\xi_0, \tau)$.  Using~\eqref{eq-bis}, we obtain
\begin{equation}\label{e:equazione2}
D\beta(\xi)\beta'(\xi)=c\frac{1-\beta(\xi)}{\eta(\xi)}-\frac{\eta'(\xi)}{\eta(\xi)}-c, \quad \forall \xi\in (-\infty, \tau).
\end{equation}
Let us fix $\xi \in [\xi_0, \tau)$. Integrating~\eqref{e:equazione2} over $[\xi_0, \xi]$, we have
\[
D\int_{\xi_0}^{\xi} \beta(s)\beta'(s) \,\mathrm{d}s=c\int_{\xi_0}^{\xi}\frac{1-\beta(s)}{\eta(s)}\,\mathrm{d}s-\ln \eta(\xi)-c(\xi-\xi_0).
\]
Since $1-\beta(\xi)< 1$ for all $\xi\in[\xi_0, \tau)$, we obtain that
\[
\frac D 2 \beta^2(\xi)-\frac D 2\beta^2(\xi_0) < c\int_{\xi_0}^{\xi}\frac{1}{\eta(s)}\, \mathrm{d}s-\ln \eta(\xi)-c(\xi-\xi_0).
\]
Moreover, from Remark~\ref{r:semi-w}, we know that $0<\beta(\xi_0)<\beta(-\infty)=1$. Hence, since $\eta(s)>1$ for all $s\in (\xi_0, \xi]$, we deduce
\[
0<\frac D2 \beta^2(\xi) <\frac D 2+ c(\xi-\xi_0)-\ln \eta(\xi)-c(\xi-\xi_0)=\frac D 2 -\ln \eta(\xi).
\]
Accordingly, we obtain $\ln \eta(\xi)<\frac D2$, for all $\xi\in[\xi_0, \tau)$. Thus, $\eta(\xi)< \sqrt{e^D}$ for all $\xi\in[\xi_0, \tau)$ and so~\eqref{eq-etad} follows.
\end{proof}

In the following result, we explore the regularity of semi-wavefronts of~\eqref{eq-sys} in $(-\infty, \tau)$. To achieve this, we will compare with solutions of~\eqref{e:z1}. We say that $\gamma\in C^1(\sigma_1,\sigma_2)$, with $0\leq\sigma_1<\sigma_2\leq1$ is a strict upper- solution of~\eqref{e:z1} in $(\sigma_1,\sigma_2)$ if $\dot{\gamma}(\beta)> -c-\frac{DN^2(\beta)\beta^2}{\gamma(\beta)}$ for all $\beta\in(\sigma_1,\sigma_2)$. We say that $\gamma$ is a strict lower-solution of~\eqref{e:z1} in $(\sigma_1,\sigma_2)$ if the inequality is reversed (see, e.g.,~\cite{CM}).

\begin{proposition}\label{p:tau}
Let $(\eta, \beta)$ be a semi-wavefront of~\eqref{eq-sys} in $(-\infty, \tau)$ for some $c>0$. Then, the following hold.
\begin{enumerate}[nosep,parsep=2pt,wide=0pt,  labelwidth=20pt, align=left]
\item[$(i)$] If $\tau=+\infty$, then $(\eta, \beta)$ is a classical wavefront of~\eqref{eq-sys}.
\item[$(ii)$] If $\tau<+\infty$, then $(\eta, \beta)$ is a wavefront of~\eqref{eq-sys} with
\begin{equation}\label{e:dopotau}
\eta(\xi)=\eta(\tau)+\frac{\eta'(\tau)}{c}\left(1-e^{c(\tau-\xi)}\right), \quad \beta(\xi)=0, \quad  \forall \xi \in[ \tau,+\infty),
\end{equation}
if and only if $\displaystyle{\lim_{\xi \to \tau^-}}D\eta(\xi)\beta(\xi)\beta'(\xi)=0$. Furthermore, if this condition is satisfied, then it is a sharp wavefront at $0$.
\end{enumerate}
\end{proposition}

\begin{proof}
We divide the proof into two parts by proving the two possible cases separately.

\smallskip
\noindent\textit{Case $(i)$}. Assume $\tau=+\infty$. From Remark \ref{r:semi-w}, we have $\eta'(\xi)>0$ and $\beta'(\xi)<0$ for all $\xi\in\mathbb{R}$. To prove~\eqref{eq-bc2}, we argue by contradiction and suppose that $\beta(+\infty)=l>0.$
From~\eqref{eq-sys-1}--\eqref{eq-bc1} and since $\eta'_0>0$, we deduce that
\[\begin{split}
\eta'(\xi)&=\eta_0'e^{-c\xi}+e^{-c\xi}\int_0^{\xi}\eta(s)\beta(s)e^{cs}\,\mathrm{d}s>e^{-c\xi}\int_0^{\xi}\eta(s)\beta(s)e^{cs}\,\mathrm{d}s\\
&>\eta_0 l e^{-c\xi}\int_0^{\xi}e^{cs}\,\mathrm{d}s = \frac{\eta_0 l}{c}(1-e^{-c\xi}).
\end{split}\]
Passing to the limit as $\xi\to+\infty$, we have
\[\lim_{\xi\to+\infty} \eta'(\xi)\ge \lim_{\xi\to+\infty} \frac{\eta_0 l}{c}(1-e^{-c\xi})= \frac{\eta_0 l}{c}>0,
\]
which is in contradiction with the boundedness of $\eta(+\infty)$ stated in Proposition~\ref{p:equilim}. Hence $\beta(+\infty)=0$. It thus remains to show that $\eta(+\infty)=1.$
According to Remark~\ref{r:semi-w} and Proposition~\ref{p:equilim}, notice that the limit $\eta(+\infty)$ exists and it is finite. Moreover, by  \eqref{e:etader} and the boundedness of $\eta$, we get also that $\eta'(+\infty)$ exists and it is necessarily $0$.
From~\eqref{eq-bis}, we deduce that
\[
\lim_{\xi \to +\infty}D\eta(\xi)\beta(\xi)\beta'(\xi)=c\left(1-\eta(+\infty)\right).
\]
If $\eta(+\infty)< 1$, then $D\eta(\xi)\beta(\xi)\beta'(\xi) > 0$ for sufficiently large values of $\xi$, which is in contradiction with the sign of the function $D\eta\beta\beta'$. On the other hand, if $\eta(+\infty) > 1$, then $\beta'(+\infty) = -\infty$, which is in contradiction with the boundedness of $\beta$. Hence $\eta(+\infty)=1$ and case $(i)$ is proved.

\smallskip
\noindent\textit{Case $(ii)$}. Assume $\tau<+\infty$ hence implying that $(\eta, \beta)$ is a semi-wavefront of~\eqref{eq-sys} in $(-\infty, \tau)$. Since $\eta \in C^1(-\infty, \tau]$, we obtain from \eqref{eq-bis} that $D\eta\beta\beta'$ has a finite limit, when $\xi \to \tau^-$ and
\[
\lambda:=\lim_{\xi \to \tau^-}D\eta(\xi)\beta(\xi)\beta'(\xi)=c-\eta'(\tau)-c\eta(\tau).
\]
Moreover, by Remark~\ref{r:semi-w}, we have $\lambda \leq 0$. Hence, $\eta'(\tau)=c-\lambda
-c\eta(\tau) $ and the function $\eta$ satisfies $\eta''+c\eta'=0$ in the half-line $(\tau, +\infty)$. Therefore
\[
\eta(\xi)=1-\frac{\lambda}{c}+\left(\frac{\lambda}{c}-1+\eta(\tau)  \right)e^{-c(\xi-\tau)}, \qquad \xi \ge \tau.
\]
As a consequence $\eta(+\infty)=1$ if and only if $\lambda=0$. In this case $(\eta, \beta)$ satisfies conditions~\eqref{e:dopotau} in the half-line $(\tau, +\infty)$.
Now we prove that, when extended as in \eqref{e:dopotau} on $(\tau, +\infty)$, the couple $(\eta, \beta)$ is a solution of~\eqref{eq-sys-1}--\eqref{eq-bc}. Indeed, $\beta \in C^2(-\infty, \tau)$ (see Remark~\ref{r:betaC2}) with $\beta(\xi)\beta'(\xi) \to 0$ as $\xi\to \tau^-$ (see Lemma~\ref{lem-1}~$(ii)$), hence $\beta\beta' \in L^1_{loc}(\mathbb{R})$. Moreover, let $\psi\in C^{\infty}_0(\tau-\epsilon,\tau+\epsilon)$ with $\epsilon>0$ and $\delta >0$ such that $\tau-\delta>\tau-\epsilon$. It holds
\[\begin{split}
\int^{\tau+\epsilon}_{\tau-\epsilon}&\left[\left(D\eta(\xi)\beta(\xi)\beta'(\xi)+c\beta(\xi)\right)\psi'(\xi) -\eta(\xi)\beta(\xi)\psi(\xi) \right]\,\mathrm{d}\xi\\
&=\lim_{\delta\to0} \int^{\tau-\delta}_{\tau-\epsilon}\left[\left(D\eta(\xi)\beta(\xi)\beta'(\xi)+c\beta(\xi)\right)\psi'(\xi) -\eta(\xi)\beta(\xi)\psi(\xi) \right]\,\mathrm{d}\xi		\\
&=\lim_{\delta\to0} \left[D\eta(\xi)\beta(\xi)\beta'(\xi)\psi(\xi)+c\beta(\xi)\psi(\xi)	\right]^{\tau-\delta}_{\tau-\epsilon}		\\
&\qquad-\lim_{\delta\to0} \int^{\tau-\delta}_{\tau-\epsilon}\left[\left(D\eta(\xi)\beta(\xi)\beta'(\xi)\right)'+c\beta'(\xi)+\eta(\xi)\beta(\xi) \right]\psi(\xi)\,\mathrm{d}\xi		\\
&=\lim_{\delta\to0} D\eta(\tau-\delta)\beta(\tau-\delta)\beta'(\tau-\delta)\psi(\tau-\delta)=D\eta(\tau)\beta(\tau)\beta'(\tau)=0.
\end{split}\]
since $D\eta\beta\beta'\to 0$ as $\xi \to \tau^-$.

It remains to show that $(\eta, \beta)$ is a sharp wavefront in $0$. To this aim, we only need to exclude the case when $\beta'(\tau^-)=0$. We reason by contradiction and assume that $\beta'(\tau^-)=0$. Let $\bar{\beta}:=\beta(\bar{\xi})>0$ for some $\bar{\xi}<\tau$. We have that
\[
\begin{array}{rl}
\displaystyle\tau=\bar{\xi}+\int_{\bar{\beta}}^{0}\dot{\xi}(\beta)\,\mathrm{d} \beta=&\bar{\xi}+\displaystyle\int_{\bar{\beta}}^{0}\frac{\mathrm{d} \beta}{\beta'(\xi(\beta))}\\
\\
\displaystyle=&\bar{\xi}+\displaystyle\int_{\bar{\beta}}^{0}\frac{D\beta N(\beta)}{z(\beta)}\,\mathrm{d} \beta=\bar{\xi}+\int_{0}^{\bar{\beta}}\frac{D N(\beta)}{-\frac{z(\beta)}{\beta}}\,\mathrm{d} \beta,
\end{array}
\]
with $\xi(\beta),$ $N(\beta)$, and $z(\beta)$ defined in~\eqref{e:Nfun} and~\eqref{e:z}.

Let $\varepsilon \in (0, c)$. We claim that there is $\beta_0 \in (0, \bar{\beta})$ such that
\begin{equation}\label{e:claim}
-\frac{z(\beta)}{\beta}<\frac{D\eta^2(\tau)}{c-\varepsilon}\beta, \quad 0<\beta<\beta_0.
\end{equation}
By \eqref{eq-bc1} and Lemma \ref{lem-1}$(v)$ we obtain that  $0<\bar{\eta}:=\eta(\bar{\xi})\leq N(\beta)\leq N(0)=\eta(\tau)<1$ for $0<\beta< \bar{\beta}$. Hence, if
\eqref{e:claim} is true then
\[
\tau>\bar{\xi}+\int_{0}^{\beta_0}\frac{D N(\beta)}{-\frac{z(\beta)}{\beta}}\,\mathrm{d} \beta>\int_{0}^{\beta_0 } \frac{N(\beta)(c-\varepsilon)}{\eta^2(\tau)\beta}\,\mathrm{d} \beta>\int_{0}^{\beta_0 }\frac{\bar{\eta}(c-\varepsilon)}{\eta^2(\tau)\beta}\,\mathrm{d} \beta=+\infty.
\]
Hence, when  \eqref{e:claim} is valid, the case  $\tau<+\infty$ and $\beta'(\tau^-)=0$ produces a contradiction, and then it is not possible. Notice that $\beta'(\tau^-)<0$ (see Lemma~\ref{l:beta'intau}) and hence the corresponding solution is sharp. Now we show \eqref{e:claim} using a comparison-type argument applied to problem~\eqref{e:zsist}. Let
\[
\gamma(\beta):=-\frac{D\eta^2(\tau)}{c-\varepsilon}\beta^2, \quad 0<\beta<1,
\]
with $\varepsilon$ as above. Since
\[
\dot{\gamma}(\beta)=-2\frac{D\eta^2(\tau)}{c-\varepsilon}\beta\to 0, \quad \text{as } \beta\to 0^+,
\]
while
\[
-c-\frac{DN^2(\beta)\beta^2}{\gamma(\beta)}=-c+(c-\varepsilon)\frac{N^2(\beta)}{\eta^2(\tau)}<-c+c-\varepsilon=-\varepsilon,
\]
there is $\tilde{\beta}\in (0,1)$ such that
\[
\dot{\gamma}(\beta)>-c-\frac{DN^2(\beta)\beta^2}{\gamma(\beta)}, \quad 0<\beta<\tilde \beta.
\]
Then, $\gamma(\beta)$ is an upper-solution to \eqref{e:z1} in $(0, \tilde \beta)$. With no loss of generality, we can assume $\tilde \beta<\bar{\beta}$. Since $\beta'(\tau^-)=0$ by the definition of $z$ in \eqref{e:z} we have that
\[
0=\lim_{\xi \to \tau^-}\beta'(\xi)=\lim_{\beta\to 0^+}\beta'(\xi(\beta))=\lim_{\beta \to 0^+}\frac{z(\beta)}{DN(\beta)\beta}
\]
implying
\[
\lim_{\beta \to 0^+}\frac{z(\beta)}{\beta}=0,
\]
since $N(\beta)\to \eta(\tau)\in (0,1)$ when $\beta \to 0$. Corresponding to a decreasing sequence  $\{\beta_n\}$ such that $\beta_n \to 0$ as $n \to +\infty$, we can find $\{\sigma_n\}$ with $0<\sigma_n<\beta_n$ for all $n$ satisfying
\[
\frac{z(\beta_n)}{\beta_n}=\dot{z}(\sigma_n)=-c-\frac{DN^2(\sigma_n)\sigma_n^2}{z(\sigma_n)}\to 0, \quad \text{as } n \to +\infty.
\]
Therefore $\displaystyle \lim_{n \to +\infty}\frac{DN^2(\sigma_n)\sigma_n^2}{-z(\sigma_n)}=c  $ and hence there is $\tilde n$ such that
\[
\frac{DN^2(\sigma_n)\sigma_n^2}{-z(\sigma_n)}>c-\varepsilon, \quad \text{for } n\ge \tilde n,
\]
again with $\varepsilon$ as above. This implies
\[
z(\sigma_n)>-\frac{DN^2(\sigma_n)\sigma_n^2}{c-\varepsilon}>-\frac{D\eta^2(\tau)\sigma_n^2}{c-\varepsilon}=\gamma(\sigma_n), \quad \text{for } n\ge \tilde n.
\]
We choose $n_0\ge \tilde n$ such that $\sigma_{n_0}<\tilde \beta$. Notice that $z(\sigma_{n_0})>\gamma (\sigma_{n_0})$; since, moreover,  $\gamma$ is and upper-solution to \eqref{e:z1} in $(0, \sigma_{n_0})\subset(0, \tilde \beta)$ we obtain (see, e.g, \cite[Lemma 4.3]{CM}) that $z(\beta)>\gamma(\beta)$ for $\beta \in (0, \sigma_{n_0})$ and \eqref{e:claim} is proved with $\beta_0=\sigma_{n_0}$.
\end{proof}

\section{Semi-wavefront on the negative half-line}\label{section-4}

In this section, we demonstrate Theorem~\ref{th-1} by proving the existence of a semi-wavefront of~\eqref{eq-sys} with positive speed $c$ in the half-line $(-\infty, 0]$ (as will be shown in Proposition~\ref{prop-ex1}) and its uniqueness (as will be shown in Proposition~\ref{prop-uniq}). 

\begin{proposition}\label{prop-ex1}
 For every $c>0$, there exists a  semi-wavefront $(\eta,\beta)$ of~\eqref{eq-sys} in $(-\infty, 0]$ with wave speed $c$. Moreover, $\eta'(\xi)>0$ and $\beta'(\xi)<0$ for all $\xi \in (-\infty, 0]$.
 \end{proposition}

The proof is based on comparison results, i.e., we use upper and lower functions (cf.,~\cite{KiSh-88}). Let $-\infty\leq a_-<a_+\leq+\infty$ and let $f\colon(a_-,a_+)\times\mathbb{R}\times\mathbb{R}\to\mathbb{R}$ be a continuous function. We say that $y$ is a \emph{lower function} of~$u''=f(\xi,u,u')$ if $y\in C^2(a_-,a_+)$ and $y''(\xi)\geq f\left(\xi,y(\xi),y'(\xi)\right)$ for every $\xi\in(a_-,a_+)$. If the $\geq$ sign is replaced by $\leq$, we say that $y$ is an \emph{upper function}.

We first investigate the solvability of~\eqref{eq-n} in $(-\infty, 0]$ with positive $c$ and sufficiently regular  $\beta$ and establish some necessary conditions for the existence of semi-wavefronts of~\eqref{eq-sys} on the half-line~$(-\infty, 0]$.

\begin{proposition}\label{prop-EXUN}
Assume $\beta\in C(-\infty, 0]$ satisfying $0<m\leq\beta(\xi)\leq M<+\infty$ for all $\xi \in (-\infty, 0]$. For every $c>0$ and $\eta_0 \in (0,1)$ equation~\eqref{eq-n} has a unique solution such that $\eta'(\xi)>0$ for all $\xi \in (-\infty, 0]$ and satisfies $\eta(0)=\eta_0$ and $\eta(-\infty)=\eta'(-\infty)=0.$
\end{proposition}

\begin{proof}
Let
\[
S_1(\xi):=\eta_0 e^{\nu_1\xi} \text{ with } \nu_1:=\frac{2m}{c+\sqrt{c^2+4m}}, \quad \text{for } \xi\in(-\infty, 0],
\]
and
\[
S_2(\xi):=\eta_0 e^{\nu_2\xi} \text{ with } \nu_2:=\frac{2M}{c+\sqrt{c^2+4M}}, \quad \text{for } \xi\in(-\infty, 0].
\]
Notice that $S_1$ is the solution of
\[\begin{cases}
\eta''(\xi)+c\eta'(\xi)-m\eta(\xi)=0,\\
\eta(0)=\eta_0, \quad \eta(-\infty)=0,
\end{cases}\]
and also an upper function of~\eqref{eq-n}. In fact, $S_1''(\xi)=-cS_1'(\xi)+mS_1(\xi)\leq-cS_1'(\xi)+\beta(\xi)S_1(\xi),$ for all $\xi\in(-\infty, 0].$ Analogously, $S_2$ is  a lower function of~\eqref{eq-n}, since $S_2''(\xi)=-cS_2'(\xi)+MS_2(\xi)\geq-cS_2'(\xi)+\beta(\xi)S_2(\xi),$ for all $\xi\in(-\infty, 0].$

Therefore, from the theory of upper and lower functions~\cite{KiSh-88}, a solution $\eta$ to~\eqref{eq-n} exists satisfying
\begin{equation}\label{eq-bound-S12}
S_2(\xi)\leq \eta(\xi)\leq S_1(\xi), \quad \forall \xi\in(-\infty, 0].
\end{equation}
Moreover, by integrating~\eqref{eq-n} in $[\xi, 0],$ with $\xi<0$, we obtain~\eqref{eq:relation1}.
Hence, $\eta'(\xi)$ has limit when $\xi \to -\infty$, and $\eta'(-\infty)=0$, since $\eta$ is bounded.  In particular, we have that
\begin{equation}\label{eq:relation2}
\eta'_0+c\eta_0=\int_{-\infty}^{0}\eta(s)\beta(s)\,\mathrm{d}s.
\end{equation}
It remains to prove that \eqref{eq-n} has a unique solution $\eta$ satisfying $\eta(\xi_0)=\eta_0$ and $\eta(-\infty)=0$. We reason by contradiction and assume   the existence of $\eta_1 \ne \eta_2$ solutions to~\eqref{eq-n} with $\eta_1(0)=\eta_2(0)=\eta_0$ and  $\eta_1(-\infty)=\eta_2(-\infty)=0$. Since $\eta_1'(0)\ne \eta_2'(0)$, we take $\eta_1'(0)>\eta_2'(0)$, so that, by \eqref{eq:relation2},
\[
0>\eta_2'(0)-\eta_1'(0)=\int_{-\infty}^{0}\left[\eta_2(s)-\eta_1(s)\right]\beta(s)\,\mathrm{d}s.
\]
Hence $\eta_2-\eta_1<0$ in some interval $(a,b)\subset(-\infty, 0)$. The relation between $\eta_2'(0)$ and $\eta_1'(0)$ implies the existence of $\xi_1<0$ such that $\eta_1(\xi_1)=\eta_2(\xi_1)$ and
\begin{equation}\label{e:2>1}
\eta_2(s)>\eta_1(s), \quad  \forall s\in (\xi_1, 0).
\end{equation}
Therefore, by applying \eqref{eq:relation1} in $(\xi_1, 0)$ to both $\eta_1$ and $\eta_2$, we have that
\[
0<\int_{\xi_1}^{0}\left[\eta_2(s)-\eta_1(s)\right]\beta(s)\,\mathrm{d}s+\eta_1'(0)-\eta_2'(0)
=-\eta_2'(\xi_1)+\eta_1'(\xi_1),
\]
in contradiction with \eqref{e:2>1}. Hence the solution of \eqref{eq-n} satisfying the boundary conditions is unique.

At last, if there is $\xi_0\leq 0$ such that $\eta'(\xi_0)=0$, by~\eqref{eq-n} we have $\eta''(\xi_0)=\eta(\xi_0)\beta(\xi_0)>0$, hence $\eta$ has a local minimum in $\xi_0$ in contradiction with the boundary condition at $-\infty$. Hence, necessarily, $\eta(\xi)>0$ for all $\xi \in (-\infty, 0].$  This concludes the proof.
\end{proof}

\begin{remark}\label{r:sigma12}
Let $(\eta,\beta)$ be a semi-wavefront of~\eqref{eq-sys} in $(-\infty, 0]$ with wave speed $c>0$. Using similar reasoning as in Proposition~\ref{prop-EXUN}, we can show that
\begin{equation}\label{eq-bound-n}
\Sigma_2(\xi):=\eta_0 e^{\sigma_2\xi}\leq\eta(\xi)\leq \Sigma_1(\xi):=\eta_0 e^{\sigma_1\xi}, \quad \forall \xi\in(-\infty, 0],
\end{equation}
where
\begin{equation}\label{eq-sigma}
\sigma_2:=\frac{2}{c+\sqrt{c^2+4}}	\quad\text{and}\quad \sigma_1:=\frac{2\beta_0}{c+\sqrt{c^2+4\beta_0}}.
\end{equation}
Moreover, since $0<\sigma_1<\sigma_2<1,$ from~\eqref{eq-bound-n}, one can deduce
\begin{equation}\label{eq-eta0}
\eta_0\sigma_1\leq\eta'_0\leq\eta_0\sigma_2.
\end{equation}
\end{remark}

\begin{proposition}\label{prop-2}
Let $(\eta,\beta)$ be a semi-wavefront of~\eqref{eq-sys} in $(-\infty, 0]$ with wave speed $c>0$ and $\eta_0\in\left(0,1\right)$, then $\beta_0\in\left(1-\frac{\eta_0(c+\sigma_2)}{c},1\right).$
\end{proposition}

\begin{proof}
We can distinguish between two cases: the first case is when $\frac{c}{c+\sigma_2}\leq\eta_0<1$, and the second case is when $0<\eta_0<\frac{c}{c+\sigma_2}$. In the first case, we see that $1-\frac{\eta_0(c+\sigma_2)}{c}\leq0$, and thus the thesis is trivially satisfied since $\beta_0>0$. As for the second case, by~\eqref{eq-bis} and the positivity of $\beta_0$, we have that
\[\beta'_0=\frac{c-c\eta_0-c\beta_0-\eta'_0}{D\eta_0\beta_0}.
\]
Since $\eta_0'\leq \eta_0\sigma_2$ (see \eqref{eq-eta0}), if  $\beta_0\leq 1-\frac{\eta_0(c+\sigma_2)}{c}$ then $\beta'_0\geq0$, which is in contradiction with $\beta'<0$ in $(-\infty, \tau)$ (see Lemma~\ref{lem-1}$(v)$).
\end{proof}

\begin{remark}\label{rem-b0}
If $\eta_0\in\left(0,\frac{c}{c+\sigma_2}\right)$, then $1-\frac{\eta_0(c+\sigma_2)}{c}>0$. By applying Proposition~\ref{prop-2}, it thus follows $\beta_0>1-\frac{\eta_0(c+\sigma_2)}{c}>0,$ and so $\tau>0$.
\end{remark}

\subsection{Existence result on the negative half-line for every speed}\label{subsection-4-1}

In this section, we fix $c>0,$ and $\eta_0 \in \left(0, \frac{c}{c+\sigma_2}\right)$, where $\sigma_2$ is defined in~\eqref{eq-sigma}. Next, we set
\begin{equation}\label{e:defb}
  \kappa:=1-\frac{\eta_0(c+\sigma_2)}{c}.
\end{equation}
Then, $\kappa>0$ and, according to Remark~\ref{rem-b0}, $\beta_0 > \kappa$ for every semi-wavefront $(\eta,\beta)$ of~\eqref{eq-sys} in $(-\infty, 0]$ with wave speed $c$. We define the following set:
\[
\mathfrak{B}:=\Big\{\beta\in C(-\infty, 0] \colon \kappa\leq \beta(\xi)\leq 1,\, \forall \xi\in(-\infty, 0]	\Big\}.
\]
It is noteworthy that $\mathfrak{B}$ is a non-empty, closed, and convex subset of the Fréchet space $C(-\infty, 0]$. Moreover, by Proposition \ref{prop-EXUN}, for every $\beta \in \mathfrak{B}$, it is unique the solution on $(-\infty, 0]$ to~\eqref{eq-n} with $\eta(0)=\eta_0$ and $\eta(-\infty)=0$ and it satisfies $\eta'(\xi)>0$ for all $\xi \in (-\infty, 0]$ and $\eta'(-\infty)=0$. We denote by $\mathfrak{N}$ the set of all such functions $\eta$, namely
\begin{equation}\label{e:N}
\mathfrak{N}:=\Big\{
\text{solution to } \begin{cases}
\eta''(\xi)+c\eta'(\xi)=\beta(\xi) \eta(\xi), \quad \xi\in(-\infty,0], \\
\eta(0)=\eta_0, \quad \eta(-\infty)=0,
\end{cases} \quad \text{with }\beta \in \mathfrak{B}\Big\}.
\end{equation}
We consider now the initial value problem
\begin{subnumcases}{\label{eq-iv}}
	y'(\xi)=\frac{c\left(1-y(\xi)\right)-\eta'(\xi)-c\eta(\xi)}{D\eta(\xi)y(\xi)}, \label{eq-iv1} \\
	y(0)=y_0, \label{eq-iv2}
\end{subnumcases}
with $\eta \in \mathfrak{N}$ and $y_0 \in (0,1)$ and show several qualitative properties of its solution before proving Proposition \ref{prop-ex1}.

\begin{proposition}\label{prop-4}
Given $\eta \in \mathfrak{N}$ and  $y_0 \in (0,1)$, let $y$ be the solution of~\eqref{eq-iv} defined on its maximal existence interval $I\subseteq(-\infty, 0]$. If there is $\xi_0\in(-\infty, 0]$ such that $y'(\xi_0)\geq0$, then $y'(\xi)>0$ for all $\xi\in I$ with $\xi<\xi_0$.
\end{proposition}

\begin{proof}
Let $\Gamma\colon\mathbb{R}\to\mathbb{R}$ be the differentiable function defined as
\[\Gamma(\xi):=c\left(y(\xi)-1\right)+\eta'(\xi)+c\eta(\xi).
\]
Notice that $\eta>0$ in $(-\infty, 0]$ and $y>0$ in $I$. Hence, by assumption, it follows that $\Gamma(\xi_0)\leq0$. Since $\Gamma'(\xi)=cy'(\xi)+\eta''(\xi)+c\eta'(\xi)$ and $\eta\in\mathfrak{N}$, we also have $\Gamma'(\xi_0)=cy'(\xi_0)+\eta(\xi_0)\beta(\xi_0)$ for some $\beta \in \mathfrak{B}$, hence $\Gamma'(\xi_0)>0$. By contradiction, let us suppose that there exists $\xi_1\in I$ with $\xi_1<\xi_0$ such that $\Gamma(\xi)<0$ for all $\xi\in(\xi_1,\xi_0)$ and $\Gamma(\xi_1)=0$. Then, we can infer that $y'(\xi)>0$ for all $\xi\in(\xi_1,\xi_0)$.
Hence, it follows that $\Gamma'(\xi)>0$, for all $\xi\in(\xi_1,\xi_0)$. Accordingly, $0=\Gamma(\xi_1)<\Gamma(\xi_2)\leq 0$, which is a contradiction.
\end{proof}

\begin{proposition}\label{prop-6}
Given $\eta\in\mathfrak{N}$ and $y_0 \in (0,1)$, let $y$ be the solution of~\eqref{eq-iv} defined on its maximal existence interval $I\subseteq(-\infty, 0]$. If there is $\xi_0\in(-\infty, 0]$ such that $y(\xi_0)=1$ then $I=(-\infty, 0]$ and $y(\xi)>1$ for all $\xi<\xi_0$.
\end{proposition}

\begin{proof}
Since $\eta\in\mathfrak{N}$ and $y(\xi_0)=1$, from~\eqref{eq-iv1} and Proposition \ref{prop-EXUN}, there exists $\beta \in \mathfrak{B}$ such that $D\eta(\xi_0)y'(\xi_0)=-\eta'(\xi_0)-c\eta(\xi_0)<0$, implying $y'(\xi_0)<0$. By contradiction, let us suppose that there exists $\xi_1\in I$ with $\xi_1<\xi_0$ such that $y(\xi)>1$ for $\xi\in(\xi_1,\xi_0]$ and $y(\xi_1)=1$.
Then, as before, we obtain that $y'(\xi_1)<0$, which is a contradiction. Hence $I=(-\infty, 0]$ and $y(\xi)>1$ for $\xi<\xi_0$.
\end{proof}

\begin{corollary}\label{prop-corbis}
Given $\eta\in\mathfrak{N}$, there is $\delta>0$ such that for any $y_0\in(1-\delta,1]$ the solution $y$ of~\eqref{eq-iv} defined on its maximal existence interval $I\subseteq(-\infty, 0]$ is such that $y(\xi_0)=1$ for some $\xi_0\in I.$
\end{corollary}

\begin{proof}
Let $\check{y}$ be the solution of~\eqref{eq-iv} associated with $y(0) = 1$ defined on its maximal existence interval $\check{I}$. Then, by Proposition~\ref{prop-6}, we have $\check{I}=(-\infty, 0]$ and $\check{y}(\xi) > 1$ for every $\xi <0$. Fix $\overline{\xi}<0$ and chose  $\epsilon > 0$ such that $\check{y}(\overline{\xi})-\epsilon > 1$. Using the continuous dependence of the solution on initial data we find a corresponding $\delta > 0$ such that for any $\rho \in (1-\delta,1)$ the solution $y$ of~\eqref{eq-iv} with $y(0) = \rho$ is defined in $[\overline{\xi}, 0]$ and satisfies $\check{y}(\xi) - \epsilon < y(\xi) < \check{y}(\xi)$ for every $\xi \in [\overline{\xi}, 0]$. Therefore, there must exist $\xi_0\in (\overline{\xi}, 0)$ where $y(\xi_0) = 1$. This completes the proof.
\end{proof}

\begin{proposition}\label{prop-y}
Let $\kappa$ be defined as in \eqref{e:defb}. For every  $\eta\in\mathfrak{N}$, there is a unique $y_0\in\left(\kappa,1\right)$ such that the solution $y$ of~\eqref{eq-iv} is defined on $(-\infty, 0]$, satisfies $y(-\infty)=1$ and $y'(\xi)<0$ for all $\xi \in (-\infty, 0]$ 
\end{proposition}

\begin{proof}
Let $\eta\in\mathfrak{N}$ be fixed. Introduce the following set:
\[
A:=\left\{y_0\in[\kappa, 1]\colon \text{$y$ solution of~\eqref{eq-iv} on $I$ satisfies } y(\xi_0)=1 \text{ for some }\xi_0\in I	\right\}.
\]
We aim to prove that $A$ is a non-empty interval such that $\alpha:=\inf A>\kappa$. This way, by denoting with $y_{\alpha}$ the solution of~\eqref{eq-iv1} with $y_0=\alpha$ defined on its maximal existence interval $I_{\alpha}\subseteq(-\infty, 0]$, we can prove that $I_{\alpha}=(-\infty, 0],$ and $y_{\alpha}'(\xi)<0$ for all $\xi \leq0$. At last $y_{\alpha}$ is the unique solution of~\eqref{eq-iv1} satisfying the boundary condition $y_{\alpha}(-\infty)=1$. We now break the proof into four steps.

\smallskip
\noindent\textit{Step 1: $A$ is non-empty and $\alpha>\kappa$}. To show that $A$ is non-empty, we observe that by Corollary~\ref{prop-corbis}, there exists some $\delta > 0$ sufficiently small such that $(1-\delta,1)\subseteq A$. Arguing as in the proof of Proposition~\ref{prop-2}, one can prove that the solution $y$ of~\eqref{eq-iv} with $y(0)=\kappa$ yields $y'(0)\geq0$. Thus, by Proposition~\ref{prop-4}, $\kappa\not\in A$.

\smallskip
\noindent\textit{Step 2: $A$ is an interval.} We can show that $A$ is an interval by proving that if $\rho\in A$, then $(\rho,1)\subseteq A$. If $\rho \in A$ and $y_{\rho}$ denotes the solution to~\eqref{eq-iv1} with $y_{\rho}(0)=\rho$, there exists $\xi_{\rho}<0$ such that $y_{\rho}(\xi_{\rho})=1$. If $\rho_1\in(\rho,1)$, the solution $y_{\rho_1}$ to~\eqref{eq-iv1} with $y_{\rho_1}(0)=\rho_1$ is defined in $[\xi_{\rho}, 1]$ and the  uniqueness of the solution of problem of~\eqref{eq-iv} guarantees the existence of $\xi_{\rho_1}\in\left(\xi_\rho,0\right)$ such that $y_{\rho_1}(\xi_{\rho_1})=1$, meaning that $\rho_1\in A$.

\smallskip
\noindent\textit{Step 3: Existence and asymptotic properties.} We aim to show that $I_{\alpha}=(-\infty, 0]$; $y_{\alpha}(-\infty)=1$; and $y_{\alpha}'(-\infty)=0$.

First, we claim that $y_{\alpha}(\xi)<1$, for every $\xi\in I_{\alpha}$. Let us suppose by contradiction that there is $\xi_0\in I_{\alpha}$ such that $y_{\alpha}(\xi_0)=1$. From Proposition~\ref{prop-6}, $y_{\alpha}(\xi)>1$, for every $\xi<\xi_0$. By arguing as in Corollary~\ref{prop-corbis} and exploiting the continuous dependence of the solutions by the initial data, we can obtain a contradiction with the definition of $\alpha$.

Next, we claim that $y_{\alpha}'(\xi)<0$, for every $\xi\in I_{\alpha}$. We argue by a contradiction and assume the existence of $\xi_1 \in (-\infty, 0]$ satisfying $y_{\alpha}'(\xi_1)\geq0$.  From Proposition~\ref{prop-4}, $y_{\alpha}'(\xi)>0$, for every $\xi \in I_{\alpha}$ with $\xi<\xi_1$. Thus, there is $\xi_2\leq\xi_1$ such that $y_{\alpha}(\xi_2)<y_{\alpha}(0)=\alpha$. By the continuous dependence of the solutions by the initial data, there exists $\delta>0$ such that for every $\rho\in(\alpha,\alpha+\delta)$ we have $y_{\alpha}(\xi_2)<y_{\rho}(\xi_2)\leq\alpha<\rho$, where $y_{\rho}$ is the solution of~\eqref{eq-iv1} with $y_{\rho}(0)=\rho$ defined on its maximal existence interval $I_{\rho}\supset (\xi_2, 0]$. By the mean value theorem, there is $\xi_3\in(\xi_2,0)$ such that $y_{\rho}'(\xi_3)>0$. Thanks to Proposition~\ref{prop-4}, we can infer that $y_{\rho}'(\xi)>0$, for every $\xi\leq\xi_3$. Then, $\rho\not\in A$, which is in contradiction with the definition of $\alpha$. It follows that $I_{\alpha}=(-\infty, 0]$; $\alpha\leq y_{\alpha}(\xi)<1,$ implying that $\alpha \not \in A$ and $y_{\alpha}'(\xi)<0$, for all $\xi\in I_{\alpha}$.

Hence, there is $y_{\alpha}(-\infty) \in [\alpha, 1].$ When, in particular,  $y_{\alpha}(-\infty) <1$ by~\eqref{eq-iv1} we get $y_{\alpha}'(-\infty)=+\infty$ in contradiction with the boundedness of $y_{\alpha}$. It follows that $y_{\alpha}(-\infty)=1.$

\smallskip
\noindent\textit{Step 4: Uniqueness.} Assume, by a contradiction, the existence of a further solution to \eqref{eq-iv1}, denoted $\hat{y}$, with $\hat{y}(-\infty)=1$. Hence, $\hat{y}(0)\in (\kappa, \alpha)$, by \textit{Step~3} and the unique solvability of the initial value problem associated to \eqref{eq-iv}. Notice that the estimate
\[
\frac{c(1-\hat{y}(0))-\eta'_0-c\eta_0}{D\eta_0\hat{y}(0)}>\frac{c(1-y_{\alpha}(0))-\eta'_0-c\eta_0}{D\eta_0y_{\alpha}(0)}
\]
is equivalent to $\eta'_0+c\eta_0<c$. By~\eqref{eq-bound-S12}, we have that $\eta_0'\leq\eta_0\nu_2$ with $M=1$, and so $\eta_0'\leq\eta_0\sigma_2$ with $\sigma_2$ as in~\eqref{eq-sigma}. This implies $\eta_0'+c\eta_0<c$ since we assumed $\eta_0 \in  (0, \frac{c}{c+\sigma_2})$. According to~\eqref{eq-iv1} we infer that $\hat{y}'(0)>y_{\alpha}'(0)$.
The function $y_{\alpha}-\hat{y}$ is then positive and decreasing in $0$. By the boundary conditions at $-\infty$ there must exist a value $\xi_0 <0$ such that
\begin{equation}\label{eq:tildehat}
y_{\alpha}(\xi_0)>\hat{y}(\xi_0)\quad \text{and} \quad y_{\alpha}'(\xi_0)=\hat{y}'(\xi_0).
\end{equation}
We show that \eqref{eq:tildehat} leads to a contradiction and then the unique solvability is proved. In fact, by  \eqref{eq-iv1},  the equality $y_{\alpha}'(\xi_0)=\hat{y}'(\xi_0)$ is equivalent to $\eta'(\xi_0)+c\eta(\xi_0)=c$ which is not possible since $\eta''(\xi)+c\eta'(\xi)=\eta(\xi)\beta(\xi)>0$ for all $\xi \in (-\infty, 0]$ implying that $\eta'+c\eta$ is an increasing function in $(-\infty, 0]$ and then $\eta'(\xi_0)+c\eta(\xi_0)<\eta'_0+c\eta_0<c$.
\end{proof}

We will now prove Proposition~\ref{prop-ex1} and establish the existence of a semi-wavefront of~\eqref{eq-sys} in the half-line~$(-\infty, 0]$ using a fixed-point argument. Before doing so, we need to recall some basic notions. We note that a subset $X$ of $C(-\infty, 0]$ is \emph{relatively compact if and only if it is bounded and their functions are pointwise equicontinuous} (see~\cite{DS-88}). Additionally, $X$ bounded means that there exists a positive continuous function $\phi\colon(-\infty, 0]\to\mathbb{R}$ such that $|f(\xi)|\leq\phi(\xi)$ for all $\xi\in(-\infty, 0]$ and $f\in X$.

\begin{proof}[Proof of Proposition~\ref{prop-ex1}]
Let us fix $c>0$, $\eta_0\in\left(0,\frac{c}{c+\sigma_2}\right)$ with $\sigma_2$ defined as in~\eqref{eq-sigma}, and $\beta\in\mathfrak{B}$. Then, we define the operator $\mathcal{T}\colon\mathfrak{B}\to C(-\infty, 0]$ as follows. First, for a given $\beta\in\mathfrak{B}$, we find the unique solution $\eta$ of problem~\eqref{e:N} whose existence and properties are established in Proposition~\ref{prop-EXUN}. Then we consider the unique solution to \eqref{eq-iv1} with $\eta$ as above, satisfying $y(-\infty)=1$; its existence is proved in Proposition~\ref{prop-y}. Formally, we write
\begin{equation}\label{eq-op}
\mathcal{T}\colon\mathfrak{B}\to C(-\infty, 0], \quad \mathcal{T}(\beta):=y,
\end{equation}
where $y$ is the solution of~\eqref{eq-iv} obtained from $\beta$ as described above. Therefore, the operator $\mathcal{T}$ is well-defined.

To prove the existence of a fixed point for $\mathcal{T}$, we will use Schauder's fixed point theorem. To this purpose, we need to show that $\mathcal{T}(\mathfrak{B})\subseteq\mathfrak{B}$ and that $\mathcal{T}$ is a continuous and compact operator with respect to the topology of $C(-\infty, 0]$. We will prove these conditions in the following steps.

\smallskip
\noindent\textit{Step 1: $\mathfrak{B}$ is invariant for $\mathcal{T}$.} This follows directly from Proposition~\ref{prop-y}, particularly given that $\alpha > \kappa$, $y'(\xi) < 0$ for all $\xi \in (-\infty, 0]$, and $y(-\infty) = 1$.

\smallskip
\noindent\textit{Step 2: $\mathcal{T}$ is compact.} Given  $\beta \in \mathfrak{B}$, let $y=\mathcal{T}(\beta)$. According to Proposition~\ref{prop-y}, we have that $\kappa\leq y(\xi)<1$ for all $\xi \in (-\infty, 0]$, implying that $\mathcal{T}(\mathfrak{B})$ is bounded in $C(-\infty, 0]$. Now we show that the functions in $\mathcal{T}(\mathfrak{B})$ are equicontinuous in every $\xi \in (-\infty, 0]$. Let $\eta$ be the unique solution to problem \eqref{e:N}. As a consequence, we notice that $\eta$ satisfies~\eqref{eq-bound-n} and in turn~\eqref{eq-eta0}. Since $\eta''(\xi)+c\eta'(\xi)=\eta(\xi)\beta(\xi)>0$, we have that $\eta'(\xi)+c\eta(\xi)<\eta_0'+c\eta_0$ for all $\xi \in  (-\infty, 0]$.  By~\eqref{eq-eta0}, we obtain that $\eta_0'+c\eta_0<\eta_0(\sigma_2+c)$.
Therefore,
\[
0<-y'(\xi)<\frac{\eta'(\xi)+c\eta(\xi)-c(1-y(\xi))}{D\eta(\xi)y(\xi)}<\frac{\eta_0(\sigma_2+c)}{D\eta(\xi)\kappa}.
\]
According to \eqref{eq-bound-n}, we then obtain
\[
0<-y'(\xi)<\frac{\sigma_2+c}{D\kappa e^{\sigma_2 \xi}}, \quad \forall\xi \in (-\infty, 0],
\]
implying that $\mathcal{T}(\mathfrak{B}) \subset C(-\infty, 0]$ is equicontinuous and so $\mathcal{T}$ is compact.

\smallskip
\noindent\textit{Step 3: $\mathcal{T}$ is continuous.} Assume that $(\beta_n)_n\subseteq\mathfrak{B}$ converges to $\beta\in\mathfrak{B}$ in $C(-\infty, 0]$ as $n\to+\infty$. Let $y_n:=\mathcal{T}(\beta_n), \, y:=\mathcal{T}(\beta)$ and denote, respectively, with  $(\eta_n)_n\subseteq\mathfrak{N}$ and $\eta\in \mathfrak{N}$ the associated solutions to the problem in~\eqref{e:N}.

We claim that
\begin{equation}\label{e:etantoeta}
\eta_n \to \eta \, \text{ in }\, C^1(-\infty, 0]\, \text{ as }\, n\to+\infty.
\end{equation}
According to Proposition~\ref{prop-EXUN} and the estimates in \textit{Step~2}, we have
\[0<\eta_n'(\xi)+\eta_n(\xi)\leq \eta_n'(\xi)+c\eta_n(\xi)+\eta_n(\xi)< \eta_0(\sigma_2 +c)+\eta_0<1+c,\]
since $\eta_0<\frac{c}{c+\sigma_2}.$ We thus obtain that $(\eta_n)_n$ is bounded in $C^1(-\infty, 0]$. 
Moreover, for all $n$, $\eta_n$ satisfies~\eqref{eq:relation1} and $\eta_n'(-\infty)=0$; therefore, by passing to the limit as $\xi \to -\infty$, we have that
\begin{equation}\label{eq:relation3}
\eta_n'(0)=\int_{-\infty}^{0}\eta_n(s)\beta_n(s)\,\mathrm{d}s-c\eta_0.
\end{equation}
Moreover, it follows
\[\begin{split}
\vert \eta_{n}''(\xi)\vert &=\vert -c\eta_{n}'(\xi)+\eta_{n}(\xi)\beta_{n}(\xi)\vert \leq c\eta_{n}'(\xi)+\eta_{n}(\xi)\beta_{n}(\xi)\leq c\eta'_{n}(\xi)+\eta_{n}(\xi)\\
&<c\eta_{n}'(\xi)+\eta_{n}(\xi)+c^2\eta_{n}(\xi)
=c\left(\eta_{n}'(\xi)+c\eta_{n}(\xi)\right)+\eta_{n}(\xi)<c^2+1.
\end{split}\]
Hence every subsequence $(\eta_{n_p})_p\subseteq (\eta_n)_n$ is equicontinuous and then every  $(\eta_{n_p})_p$ is relatively compact in $C^1(-\infty, 0]$.
Therefore, we can extract a subsequence, for simplicity denoted as the sequence, which converges to $\hat{\eta}$ in $C^1(-\infty, 0]$.  If we restrict the estimate \eqref{eq:relation1} to the sequence $(\eta_{n_p}')_p$ and pass to the limit when $p \to +\infty$ we obtain that $\hat{\eta}$ satisfies the differential equation in~\eqref{e:N}, for every $\xi \in (-\infty, 0]$. Moreover, $\eta'_{n_p}(0)$ satisfies~\eqref{eq:relation3}, and so we can pass to the limit as  $p \to +\infty$, since since $0<\eta_{n}(\xi)\beta_n(\xi)<\eta_0e^{\nu_1 \xi}$ for $\xi \in (-\infty, 0]$ and $n$ arbitrary (see~\eqref{eq-bound-S12} with $m=\kappa$). We thus get
\[
\hat{\eta}'(0)=\int_{-\infty}^{0}\hat{\eta}(s)\beta(s)\,\mathrm{d}s-c\eta_0,
\]
implying that $\hat{\eta}'(\xi)+c\hat{\eta}(\xi)\to 0$, as $\xi \to -\infty$. As a consequence, $\hat{\eta}(-\infty)=0$, namely $\hat{\eta} \in \mathfrak{N}$. Hence, by the uniqueness of the solutions to the problem in~\eqref{e:N}, it follows $\hat{\eta}=\eta$. By the arbitrariness of the subsequence $(\eta_{n_p})_p$, we proved \eqref{e:etantoeta}.

Now, we consider the sequence $(y_n)_n$. Notice that $\kappa\leq y_n(\xi)\leq 1$ and $y_n'(\xi)<0$ for all $\xi \in (-\infty, 0]$ and $n \in \mathbb{N}$. Since $\mathcal{T}(\mathfrak{B})$ is relatively compact in $C(-\infty, 0]$ (see \textit{Step~2}), from every  $(y_{n_p})_p\subseteq (y_n)_n$ we can extract a subsequence, again denoted $(y_{n_p})_p$, which converges in $C(-\infty, 0]$ to some $\hat{y}$. We prove now that $y=\hat{y}$. In fact, by passing to the limit in
\[
y_{n_p}(0)-y_{n_p}(\xi)=\int_{\xi}^{0} \frac{c(1-y_{n_p}(s))-\eta_{n_p}'(s)-c\eta_{n_p}(s)}{D\eta_{n_p}(s)y_{n_p}(s)}\,\mathrm{d}s
\]
as $p \to +\infty$ we have that $\hat{y}$ satisfies the estimate
\[
\hat{y}(0)-\hat{y}(\xi)=\int_{\xi}^{0} \frac{c(1-\hat{y}(s))-\eta'(s)-c\eta(s)}{D\eta(s)\hat{y}(s)}\,\mathrm{d}s
\]
and so $\hat{y}$ is a solution to equation~\eqref{eq-iv1} on the half-line $(-\infty, 0]$ with $\hat{y}'(0)\leq 0$.
By Proposition~\ref{prop-4}, if $\hat{y}'(0)=0$, we arrive to the contradictory conclusion  that $\hat y$ is strictly increasing in $(-\infty, 0]$. Hence $\hat{y}'(0)<0$  and by Proposition~\ref{prop-y},  $\hat{y}(-\infty)=1$ and $\hat{y}=y$. The arbitrariness of the subsequence $(y_{n_p})_p$ then implies that $(y_{n})_n$ converges to $y$ in $C(-\infty, 0]$ and $\mathcal{T}$ is continuous.
\end{proof}

\subsection{Uniqueness result on the negative half-line for every speed}\label{subsection-4-2}

\begin{proposition}\label{prop-uniq}
For every $c>0$ there is at most one (up to shift) semi-wavefront $(\eta,\beta)$ of~\eqref{eq-sys} in $(-\infty, 0]$ with wave speed $c$.
\end{proposition}

\begin{remark}[Reduction to a desingularized first-order problem]
For every $c>0$ and $\eta_0 \in \left(0, \frac{c}{c+\sigma_2}\right)$, let $(\eta,\beta)$ be a semi-wavefront of~\eqref{eq-sys} in $(-\infty, 0]$ with wave speed $c$.
We consider the function $u=\Phi(\xi)$ such that
\[\begin{cases}
\Phi'(\xi)=\dfrac{1}{D\eta(\xi)\beta(\xi)}, \quad \xi\in(-\infty,0],\\
\Phi(0)=0.
\end{cases}\]
Notice that the function $\Phi$ is a diffeomorphism with $\Phi(-\infty)=-\infty$ because $\Phi'(\xi)>0$ for all $\xi\in(-\infty, 0]$.
Hence, we can define
\[\begin{split}
&p(u):=-\frac{\eta'\left(\Phi^{-1}(u)\right)}{c}-\eta\left(\Phi^{-1}(u)\right), \\
&q(u):=\eta\left(\Phi^{-1}(u)\right), \\
&r(u):=\beta\left(\Phi^{-1}(u)\right)+\eta\left(\Phi^{-1}(u)\right)+\frac{\eta'\left(\Phi^{-1}(u)\right)}{c}-1.
\end{split}\]
Then, for every $u\in(-\infty,0]$, $(p,q,r)$ solves the following first-order problem
\begin{equation}\label{eq-manifold}
\begin{cases}
\dot p(u)=-\frac{D}{c}q^2(u)\left(p(u)+r(u)+1\right)^2,\\
\dot q(u)=cD\left(-p(u)-q(u)\right)\left(p(u)+r(u)+1\right) q(u),\\
\dot r(u)=-c r(u)+\frac{D}{c}q^2(u)\left(p(u)+r(u)+1\right)^2.
\end{cases}
\end{equation}
Since $\Phi^{-1}(-\infty)=-\infty$, it follows that $(p,q,r)$ satisfies the conditions
\begin{equation}\label{equilibrium}
p(-\infty)=0, \quad q(-\infty)=0 \quad r(-\infty)=0.
\end{equation}
Moreover, by Lemma~\ref{lem-1} and Remark~\ref{r:semi-w}, we deduce that $p<0$, $0<q<1$, $-p-q>0$, and $p+r+1>0$ in $(-\infty,0]$; hence, $\dot q>0$ in $(-\infty,0]$. 

On the other hand, since $\Phi$ is a diffeomorphism, we can consider the inverse function $\xi=\Phi^{-1}(u)$ in $(-\infty,0]$. Accordingly, if the triple of function $(p,q,r)$ satisfies~\eqref{eq-manifold}--\eqref{equilibrium}, then $\eta(\xi)=q\left(\phi(\xi)\right)$ and $\beta(\xi)=p\left(\phi(\xi)\right)+r\left(\phi(\xi)\right)+1$ satisfy~\eqref{eq-sys-1}--\eqref{eq-bc1}.
\end{remark}

\begin{proof}[Proof of Proposition~\ref{prop-uniq}]
Let $c>0$ be fixed. It is easy to show that $(0,0,0)$ is an equilibrium for~\eqref{eq-manifold}. By evaluating the Jacobian matrix $J$ at this equilibrium, we obtain
\[
J(0,0,0)=\left[\begin{array}{ccc}
0  &  0 & 0\\
0  &  0 & 0\\
0  &  0 & -c \\
\end{array}\right].
\]
Consequently, the eigenvalues of $J(0,0,0)$ are $\lambda_1=\lambda_2=0$ and $\lambda_3=-c$, with corresponding eigenvectors $e_1=(1,0,0)$, $e_2=(0,1,0)$ and $e_3=(0,0,1)$. Thanks to the Center Manifold Theorem (see~\cite[Theorem~3.2.1]{GuHo-83}), it follows the existence of a possibly non-unique center manifold whose tangent space is generated by $\{e_1,e_2\}$, and a unique stable invariant manifold tangent to the eigenspace generated by $e_3$. The center manifold has two dimensions, while the stable manifold has only one dimension.

Notice that a solution of~\eqref{eq-manifold} satisfying~\eqref{equilibrium} must leave the equilibrium $(0,0,0)$ via the center manifold. Therefore, by following the approach in~\cite{GuoTsai-01}, we prove that the center manifold is unique. Let us consider two center manifolds $W_1$ and $W_2$ such that
\[
W_1=\left(p,q,k_1(p,q)\right), \quad W_2=\left(p,q,k_2(p,q)\right)
\]
for some $C^1$ functions $k_1,k_2$ such that $k_i(0,0)=\frac{\partial}{\partial p}k_i(0,0)=\frac{\partial}{\partial q}k_i(0,0)=0$, with $i\in\{1,2\}.$ Since $W_1$ and $W_2$ are invariant for the flow associated with~\eqref{eq-manifold}, we have
\[\begin{split}
&\frac{\partial k_i}{\partial p}(p(u),q(u))\Big(-\frac{D}{c}q(u)^2\left(p(u)+k_i(p(u),q(u))+1\right)^2	\Big)\\
&\quad+\frac{\partial k_i}{\partial q}(p(u),q(u))\Big( cD\left(-p(u)-q(u)\right)\left(p(u)+k_i(p(u),q(u))+1\right) q(u) \Big)\\
&=-c k_i(p(u),q(u))+\frac{D}{c}q(u)^2\left(p(u)+k_i(p(u),q(u))+1\right)^2,
\quad i\in\{1,2\}.
\end{split}\]
Let us consider the bounded function
\[
\delta(u)=k_1(p(u),q(u))-k_2(p(u),q(u)).
\]
Then, it follows
\[
\dot{\delta}(u)+c\delta(u)=\frac{D}{c}q(u)^2\left[	\left(p(u)+k_1(p(u),q(u))+1\right)^2-\left(p(u)+k_2(p(u),q(u))+1\right)^2	\right].
\]
By integrating in $(-\infty,u]$, we deduce
\[\begin{split}
\delta(u)&=\int_{-\infty}^u \frac{D}{c} e^{c(s-u)} q(s)^2\left[	\left(p(s)+k_1(p(s),q(s))+1\right)^2-\left(p(s)+k_2(p(s),q(s))+1\right)^2	\right]\,\mathrm{d}s\\
&=\int_{-\infty}^u \frac{D}{c} e^{c(s-u)} \delta(s) q(s)^2\left[2p(s)+2+k_1(p(s),q(s))+k_2(p(s),q(s))\right]\,\mathrm{d}s.
\end{split}\]
Since, $q(u)^2\left(2p(u)+2+k_1(p(u),q(u))+k_2(p(u),q(u)\right)\to0$ as $u\to -\infty$, then there is $\bar{u}<0$ such that for all $u<\bar{u}$, we have $q(u)^2\left(2p(u)+2+k_1(p(u),q(u))+k_2(p(u),q(u)\right)<\frac{1}{2D}$. Therefore, for all $u<\bar{u}$, we deduce
\[
|\delta(u)|\leq \frac{1}{2}\int_{-\infty}^u |\delta(s)| \frac{e^{c(s-u)}}{c}\,\mathrm{d}s\leq \frac{1}{2}\sup_{u\in(-\infty,\bar{u}]}|\delta(u)|.
\]
Then, we have $|\delta(u)|=0$ for all $u<\bar{u}$, and so $r(u)=k_1(p(u),q(u))=k_2(p(u),q(u))$ which lead to $W_1=W_2$. In the following, we denote the unique central manifold by $k(p(u), q(u))$.

It remains to show that the system
\[
\begin{cases}
\dot{p}(u) = -\frac{D}{c} q(u)^2 \left( p(u) + k(p(u), q(u)) + 1 \right)^2, \\
\dot{q}(u) = cD \left( -p(u) - q(u) \right) \left( p(u) + k(p(u), q(u)) + 1 \right) q(u),
\end{cases}
\]
has at most one solution satisfying $p(-\infty) = 0$ and $q(-\infty) = 0$. To proceed, we recall that $\dot{q} > 0$. Thus, due to the strict monotonicity of $q$, we can define the inverse function $u = u(q)$ of $q$ on $(-\infty, \bar{u})$ and express $p$ as a function of $q$, i.e., $p = p(q)$. We then need to prove that
\begin{equation}\label{syst-pq-xx}
\dot{p}(q)= \frac{q\left(p(q) + k(p(q), q) + 1\right)}{c^2\left(p(q) + q\right)}
\end{equation}
has at most one solution satisfying $\lim_{q\to0^+}p(q) = 0$.

Assume, by contradiction, that $p_1$ and $p_2$ are two solutions of~\eqref{syst-pq-xx} defined on the same interval $(0, q_0)\subset (0,1)$ such that
\[
\lim_{q\to0^+}p_1(q) = 0=\lim_{q\to0^+}p_2(q).
\]
Without loss of generality, suppose $p_2(q_0) < p_1(q_0)$. By the uniqueness of solutions of the Cauchy problem associated with~\eqref{syst-pq-xx}, it follows that $p_1(q) - p_2(q) > 0$, for all $q\in(0,q_0)$. On the other hand, we compute
\[\begin{split}
\dot p_1(q)-&\dot p_2(q)=
\frac{q}{c^2}\left[ \frac{\left(p_1(q) + k\left(p_1(q), q\right) + 1\right)}{p_1(q) + q}-\frac{\left(p_2(q)+ k\left(p_2(q), q\right) + 1\right)}{p_2(q) + q}	\right]\\
&=\frac{q}{c^2}\left[\frac{	\left(p_2(q)-p_1(q)\right)(1-q) + k\left(p_1(q),q\right)\left(p_2(q)+q\right) - k\left(p_2(q),q\right)\left(p_1(q)+q\right)}{\left(p_1(q) + q\right)\left(p_2(q) + q\right)}\right].
\end{split}\]
By adding the term $\pm(p_2(q) + q)k_2(p_2(q), q)$ to the numerator on the left-hand side and dividing both sides by $\left(p_1(q) - p_2(q)\right)$, we obtain
\[
\frac{\dot p_1(q)-\dot p_2(q)}{p_1(q)-p_2(q)}=A(q)\left[q-1-k\left(p_2(q), q\right))+ \left(p_2(q)+q\right)\frac{k\left(p_1(q), q\right)-k\left(p_2(q), q\right)}{p_1(q)-p_2(q)}
\right],
\]
where
\[A(q)=\frac{q}{c^2\left(p_1(q) + q\right)\left(p_2(q) + q\right)\left(p_1(q)-p_2(q) \right)}>0, \quad \forall q\in(0,1).
\]

We claim that
\[
\lim_{q\to0^+}q-1-k\left(p_2(q), q\right))+ \left(p_2(q)+q\right)\frac{k\left(p_1(q), q\right)-k\left(p_2(q), q\right)}{p_1(q)-p_2(q)}=-1.
\]
Once the claim is proved, since $p_1(q) - p_2(q) > 0$ for every $q \in (0, q_0)$, it follows that there is $q_1 \in (0, q_0)$ such that $\dot{p}_1(q) - \dot{p}_2(q) < 0$ for all $q \in (0, q_1)$. This implies that $p_1 - p_2$ is decreasing in $(0, q_1)$ and so $p_1(q)-p_2(q)>p_1(q_1)-p_2(q_1)$ for $q<q_1$, leading to a contradiction since $\lim_{q\to0^+}p_1(q) - p_2(q) = 0$. 

To prove the claim, notice that
\[
\lim_{q \to 0^+} (q - 1 - k(p_2(q), q)) = -1,
\]
since $k(0, 0) = 0$. Given that $p_2(q) + q < 0$, for all $q\in(0,1)$, it remains to prove that
\begin{equation}\label{eq-lim-var}
\lim_{q \to 0^+} \frac{k(p_1(q), q) - k(p_2(q), q)}{p_1(q) - p_2(q)} = 0.
\end{equation}
Let $\varepsilon > 0$ be given. Since $\frac{\partial}{\partial p} k(0, 0) = 0$, there exists $\delta = \delta(\varepsilon) > 0$ such that for all $|p| < \delta$ and $|q| < \delta$, we have
\begin{equation}\label{ineq-k}
\left|\frac{\partial}{\partial p} k(p, q)\right| < \varepsilon.
\end{equation}
Given that $\lim_{q \to 0^+} p_1(q) = 0 = \lim_{q \to 0^+} p_2(q)$, there exists $\delta_1 = \delta_1(\varepsilon) > 0$ with $\delta_1<\delta$ such that $|q| < \delta_1$ implies $|p_i(q)|\leq \delta$, for $i\in\{1,2\}$. Therefore, from \eqref{ineq-k}, it follows that
\[
\left| \frac{k(p_1(q), q) - k(p_2(q), q)}{p_1(q) - p_2(q)} \right| = \left| \frac{\partial}{\partial p} k(p(q), q) \right| < \varepsilon,
\]
for some $p(q)$ between $p_1(q)$ and $p_2(q)$ by the Mean Value Theorem. This proves \eqref{eq-lim-var} and completes the proof.
\end{proof}

\begin{proof}[Proof of Theorem~\ref{th-1}] The thesis follows by Propositions~\ref{prop-ex1} and~\ref{prop-uniq}.
\end{proof}

\section{Existence of wavefront solutions}\label{section-5}
In this section, we aim to demonstrate Theorem~\ref{th-2}.
To proceed with the proof of the theorem, we first consider the semi-wavefront $(\eta, \beta)$ of~\eqref{eq-sys} in $(-\infty, 0]$ with wave speed $c>0$, whose unique existence is established by Theorem~\ref{th-1}. Next, we extend this semi-wavefront to the interval $[0, \tau)$, where $\tau$ is defined as in~\eqref{eq-tau}. Notably, $\eta'>0$ and $\beta'<0$ hold on $(-\infty, \tau)$ (see Lemma~\ref{lem-1} and Remark~\ref{r:semi-w}), allowing us to define $z_c(\beta)$ for $0<\beta<1$, as described in~\eqref{e:z}, and consider problem~\eqref{e:zsist}.

Secondly, we introduce the following auxiliary first-order boundary value problem:
\begin{subnumcases}{\label{eq-w}}
\dot w(\beta)=-\sigma-\frac{g(\beta)}{w(\beta)}, \quad 0<\beta<2, \label{eq-w1}\\
\dot w(\beta)<0, 	\quad 0<\beta<2, \label{eq-w2}\\
w(0)=w(2)=0, \label{eq-w3}
\end{subnumcases}
where $\dot w=\frac{\mathrm{d}w}{\mathrm{d}\beta}$ and $g\colon [0, 2] \to \mathbb{R}$ is a continuous function satisfying: $g(s)=De^Ds^2$ for all $s\in [0, 1]$, $0<{g(s)}/{s}<De^D$ for all $s \in (1,2)$, and $g(2)=0$. We denote with  $w_{\sigma}$ the solution of~\eqref{eq-w} associated with the parameter $\sigma$ and recall the following result.

\begin{proposition}\label{p:sistw}(see~\cite[Theorem 2.6]{CM} and~\cite[Lemma 2.1]{MaMa-05})
There exists a positive constant $\sigma^*$ satisfying $0<\sigma^*\leq 2\sqrt{De^D}$ such that problem~\eqref{eq-w} has a unique solution if and only if $\sigma\geq \sigma^*$. Moreover $\dot w_{\sigma^*}(0)=-\sigma^*$ and $\dot w_{c}(0)=0$ for $c>\sigma^*$.
\end{proposition}

Lastly, in the proof of Theorem~\ref{th-2}, we make use of the following comparison result between solutions of problems~\eqref{e:zsist} and~\eqref{eq-w}. For what follows, we recall the definitions of the values $\eta_0, \eta_0'$, $\beta_0$, and $\beta_0'$ in~\eqref{eq-0}.

\begin{proposition}\label{p:zbeta0}
Let $w_{\sigma^*}$ be the solution of problem~\eqref{eq-w} associated with the parameter $\sigma^*$. Then, for every $c\ge\sigma^*$ the solution $z_c$ of~\eqref{e:zsist}  satisfies
\begin{equation}\label{e-zbeta0}
w_{\sigma^*}(\beta)<z_c(\beta) <0,\quad \forall \beta\in(0,\beta_0].
\end{equation}
\end{proposition}

\begin{proof}
Using the threshold speed $\sigma^*$ from Proposition~\ref{p:sistw} for problem~\eqref{eq-w}, let us define
\[
\mu:=- \max_{\tilde\beta \in \left[\frac 12, 1\right]} w_{\sigma^*}(\tilde\beta).
\]
We observe that $\mu<\frac{\sigma^*}{2}$. Indeed, from~\eqref{eq-w1}, we obtain that
\[
\dot{w}_{\sigma^*}(\beta)=-\sigma^*+\frac{g(\beta)}{-w_{\sigma^*}(\beta)}>-\sigma^*, \quad 0<\beta<2.
\]
Therefore, by integrating in $[0,\beta]$, we obtain
\[
w_{\sigma^*}(\beta)=w_{\sigma^*}(\beta)-w_{\sigma^*}(0)=\int_{0}^{\beta}\dot{w}_{\sigma^*}(s)\,\mathrm{d}s>-\sigma^*\beta,
\]
which implies $\max_{\beta \in \left[\frac 12, 1\right]} w_{\sigma^*}(\beta)>-\frac{\sigma^*}{2}.$

We take $c \geq \sigma^*$ and set
\begin{equation}\label{eq-etaeta}
\eta_0 = \eta_0(c) := \frac{\mu}{c + \sigma_2},
\end{equation}
with $\sigma_2$ defined as in~\eqref{eq-sigma}.
First of all, we notice that $\eta_0 \in \left(0, \frac{c}{c+\sigma_2}\right)$ since $\mu <\frac{\sigma^*}{2}<c$. Then, from Proposition~\ref{prop-2}, we obtain $\beta_0> 1-\frac{\eta_0(c+\sigma_2)}{c}=1-\frac{\mu}{c}>1-\frac{\sigma^*}{2c}.$ Thus, since $c\ge\sigma^*$,
\begin{equation}\label{eq-bb0}
\frac{1}{2}<\beta_0<1.
\end{equation}

Next, we divide the proof into two parts. First, we claim that $z_c(\beta_0) >w_{\sigma^*}(\beta_0)$. To this purpose, we consider $z_c(\beta_0)=DN(\beta_0)\beta_0\beta'\left(\xi(\beta_0)\right)=D\eta_0\beta_0\beta'_0,$ 
and, by using formula~\eqref{eq-bis}, we obtain
\[
z_c(\beta_0)=D\eta_0\beta_0\beta'_0=c(1-\beta_0)-\eta_0'-c\eta_0>-\eta_0'-c\eta_0.
\]
Since $\eta_0'\leq \eta_0\sigma_2$ via~\eqref{eq-eta0}, we get
\[
z_c(\beta_0)>-\eta_0'-c\eta_0\geq -\eta_0(c+\sigma_2)=-\frac{\mu}{c+\sigma_2}(c+\sigma_2)=-\mu.
\]
Hence,
\[
z_c(\beta_0)>-\mu =\max_{\beta \in \left[\frac{1}{2}, 1\right]} w_{\sigma^*}(\beta)\geq w_{\sigma^*}(\beta_0),
\]
where the last inequality follows from~\eqref{eq-bb0}. Thus, the claim is proved.

As for the second part, we suppose by contradiction that there exists a $\beta_1 \in (0, \beta_0)$ such that $z_c(\beta_1)=w_{\sigma^*}(\beta_1)$. Without loss of generality, we can assume that $z_c(\beta)>w_{\sigma^*}(\beta)$ for $\beta \in (\beta_1, \beta_0]$. Since $\beta_1>0$, by Proposition \ref{p:equilim} we have
\[N(\beta_1)=\eta(\xi(\beta_1))<\eta(\tau)\leq  \sqrt{e^D}.
\]
Therefore, $DN^2(\beta_1)\beta_1^2<De^D\beta_1^2=g(\beta_1)$. Since $c\ge \sigma^*$,
\[\dot z_c(\beta_1)-\dot{w}_{\sigma^*}(\beta_1)=-c+\frac{DN^2(\beta_1)\beta_1^2}{-z_c(\beta_1)} +\sigma^*-\frac{g(\beta_1)}{-w_{\sigma^*}(\beta_1)}
<-c+\sigma^*\leq 0.
\]
Therefore, there exists $\varepsilon >0$ such that $z_c(\beta)-w_{\sigma^*}(\beta)<0$, for all $\beta \in (\beta_1, \beta_1+\varepsilon)$. This contradicts the condition $z_c>w_{\sigma^*}$ in $(\beta_1, \beta_0]$. At last the existence of $\beta_2 \in (0, \beta_0)$ such that $z_c(\beta_2)=0$ implies $\dot z_c(\beta_2)=+\infty$, in contradiction with $z_c<0$ in $(\beta_2, \beta_0)$. Hence, \eqref{e-zbeta0} is satisfied, and the proof is complete.
\end{proof}

\begin{proof}[Proof of Theorem~\ref{th-2}]
Let us define $c_0:=\sigma^*,$ where $\sigma^*$ is the threshold speed associated to~\eqref{eq-w} (see Proposition~\ref{p:sistw}). Consider the semi-wavefront $(\eta, \beta)$ of~\eqref{eq-sys} in $(-\infty, 0]$ with wave speed $c\geq c_0$ (see Theorem~\ref{th-1}) and extend it to~$[0,\tau)$.

We consider two cases for $\tau$. If $\tau=+\infty$, then $(\eta, \beta)$ is a classical wavefront of~\eqref{eq-sys} with wave speed $c$ by Proposition~\eqref{p:tau}~$(i)$. On the contrary, if $\tau<+\infty$, we consider the function $z_c$ defined in~\eqref{e:z} which satisfies problem~\eqref{e:zsist}. By Proposition~\eqref{p:zbeta0} and Lemma~\ref{lem-1}, we have that
 \[
 0\ge \lim_{\xi \to \tau^-}D\eta(\xi)\beta(\xi)\beta'(\xi)=\lim_{\beta\to 0^+}z_c(\beta)\ge \lim_{\beta\to 0^+}w_{\sigma^*}(\beta)=0;
 \]
hence, $(\eta, \beta)$ is a wavefront of~\eqref{eq-sys} with wave speed $c$ by Proposition~\ref{p:tau}~$(ii)$.

We now show that $\tau<+\infty$ is impossible when $c>c_0$. Indeed, let $c>c_0$. Notice that $\dot z_c(0)$ exists (see \cite[Lemma 2.1]{MaMa-05}) and it is $\dot z_c(0)=0$ or $\dot z_c(0)=-c$. Since $\dot w_{\sigma^*}=-\sigma^*$ (see Proposition \ref{p:sistw}) the latter case is in contradiction with  inequality \eqref{e-zbeta0}. Hence $\dot z_c(0)=0$ and then, by \eqref{e:z}, we obtain
\[
\lim_{\xi \to \tau^-}\beta'(\xi)=\lim_{\beta\to0^+} \frac{z_c(\beta)}{D\beta\eta(\xi(\beta))}=0,
\]
which is not compatible with $\tau<+\infty$ (as showed in the proof of Proposition~\ref{p:tau}~$(ii)$).

Finally, we will refer to Propositions~\ref{est-c} and~\ref{p:sistw} for comparison estimates of $c_0$. This completes the proof.
\end{proof}
%
%

{
\bibliographystyle{elsart-num-sort}
\bibliography{biblio_agar}

\begin{thebibliography}{10}
\expandafter\ifx\csname url\endcsname\relax
  \def\url#1{\texttt{#1}}\fi
\expandafter\ifx\csname urlprefix\endcsname\relax\def\urlprefix{URL }\fi

\bibitem{AiHuang-07}
S.~Ai, W.~Huang, Travelling wavefronts in combustion and chemical reaction
  models, Proc. Roy. Soc. Edinburgh Sect. A 137~(4) (2007) 671--700.

\bibitem{BNS-85}
H.~Berestycki, B.~Nicolaenko, B.~Scheurer, Traveling wave solutions to
  combustion models and their singular limits, SIAM J. Math. Anal. 16~(6)
  (1985) 1207--1242.

\bibitem{BCM1}
D.~Berti, A.~Corli, L.~Malaguti, Uniqueness and nonuniqueness of fronts for
  degenerate diffusion-convection reaction equations, Electron. J. Qual. Theory
  Differ. Equ. No. 66, pp. 1--34.

\bibitem{BiNe-91}
J.~Billingham, D.~J. Needham, The development of travelling waves in quadratic
  and cubic autocatalysis with unequal diffusion rates. {I}. {P}ermanent form
  travelling waves, Philos. Trans. Roy. Soc. London Ser. A 334~(1633) (1991)
  1--24.

\bibitem{Bo-95}
A.~Bonnet, Non-uniqueness for flame propagation when the {L}ewis number is less
  than {$1$}, European J. Appl. Math. 6~(4) (1995) 287--306.

\bibitem{CM}
A.~Corli, L.~Malaguti, Semi-wavefront solutions in models of collective
  movements with density-dependent diffusivity, Dyn. Partial Differ. Equ.
  13~(4) (2016) 297--331.

\bibitem{DS-88}
N.~Dunford, J.~T. Schwartz, Linear operators. {P}art {I}, Wiley Classics
  Library, John Wiley \& Sons, Inc., New York, 1988, general theory, With the
  assistance of William G. Bade and Robert G. Bartle, Reprint of the 1958
  original, A Wiley-Interscience Publication.

\bibitem{GiKe-04}
B.~H. Gilding, R.~Kersner, Travelling waves in nonlinear diffusion-convection
  reaction, vol.~60 of Progress in Nonlinear Differential Equations and their
  Applications, Birkh\"{a}user Verlag, Basel, 2004.

\bibitem{GKCB-98}
I.~Golding, Y.~Kozlovsky, I.~Cohen, E.~Ben-Jacob, Studies of bacterial
  branching growth using reaction-diffusion models for colonial development,
  Phys. A 260~(3) (1998) 510--554.

\bibitem{GuHo-83}
J.~Guckenheimer, P.~Holmes, Nonlinear oscillations, dynamical systems, and
  bifurcations of vector fields, vol.~42 of Applied Mathematical Sciences,
  Springer-Verlag, New York, 1983.

\bibitem{GuoTsai-01}
J.-S. Guo, J.-C. Tsai, Traveling waves of two-component reaction-diffusion
  systems arising from higher order autocatalytic models, Quarterly of Applied
  Mathematics 67~(3) (2009) 559--578.

\bibitem{Ha-04}
D.~Hartmann, Pattern formation in cultures of bacillus subtilis, Journal of
  Biological Systems 12~(02) (2004) 179--199.

\bibitem{KMUS-97}
K.~Kawasaki, M.~Matsushita, T.~Umeda, N.~Shigesada, Modeling spatio-temporal
  patterns generated by bacillus subtilis, J. Theor. Biol. 188 (1997) 177--185.

\bibitem{KiSh-88}
I.~Kiguradze, B.~Shekhter, Singular boundary-value problems for ordinary
  second-order differential equations., J. Math. Sci. 43 (1988) 2340--2417.

\bibitem{Lo-97}
E.~Logak, Mathematical analysis of a condensed phase combustion model without
  ignition temperature, Nonlinear Anal. 28~(1) (1997) 1--38.

\bibitem{MaMa-05}
L.~Malaguti, C.~Marcelli, Finite speed of propagation in monostable degenerate
  reaction-diffusion-convection equations, Adv. Nonlinear Stud. 5.

\bibitem{Ma-85}
M.~Marion, Qualitative properties of a nonlinear system for laminar flames
  without ignition temperature, Nonlinear Anal. 9~(11) (1985) 1269--1292.

\bibitem{MSM-00}
M.~Mimura, H.~Sakaguchi, M.~Matsushita, Reaction–diffusion modelling of
  bacterial colony patterns, Phys. A 282~(1) (2000) 283--303.

\bibitem{MHST-24}
E.~Mu\~noz Hern\'andez, E.~Sovrano, V.~Taddei, Coupled reaction-diffusion
  equations with degenerate diffusivity: wavefront analysis, arXiv:2311.05385
  (2024).

\bibitem{SMGA-01}
R.~A. Satnoianu, P.~K. Maini, F.~S. Garduno, J.~P. Armitage, Travelling waves
  in a nonlinear degenerate diffusion model for bacterial pattern formation,
  Discrete Contin. Dyn. Syst. Ser. B 1~(3) (2001) 339--362.

\end{thebibliography}
\bigskip\Addresses
}
\end{document}